\documentclass[review,3p,authoryear,square]{elsarticle}
\usepackage[T1]{fontenc}
\usepackage{amssymb}
\usepackage{amsmath}
\usepackage{amsthm}

\usepackage{fancyhdr}
\usepackage{hyperref}
\hypersetup{%
  pdfauthor={Lishun XIAO},%
  pdfstartview=FitH,%
  CJKbookmarks=true,%
  bookmarksnumbered=true,%
  bookmarksopen=true,%
  colorlinks=true, linkcolor=blue, urlcolor=blue, citecolor=blue, %
}

\makeatletter
\def\ps@pprintTitle{%
	\let\@oddhead\@empty
	\let\@evenhead\@empty
	\def\@oddfoot{\footnotesize\itshape
		\hfill\today}%
	\let\@evenfoot\@oddfoot}
\makeatother

\newcommand{\rtn}{\mathrm{\mathbf{R}}}
\newcommand*{\PR}{\mathrm{\mathbf{P}}}
\newcommand*{\EX}{\mathrm{\mathbf{E}}}
\newcommand*{\dif}{\mathop{}\!\mathrm{d}}

\newcommand{\calA}{\mathcal{A}}
\newcommand{\calB}{\mathcal{B}}

\newcommand{\calF}{\mathcal{F}}
\newcommand{\calH}{\mathcal{H}}

\newcommand{\calM}{\mathcal{M}}

\newcommand{\calO}{\mathcal{O}}
\newcommand{\calP}{\mathcal{P}}
\newcommand{\calS}{\mathcal{S}}
\newcommand{\calT}{\mathcal{T}}
\newcommand{\calU}{\mathcal{U}}
\newcommand{\calV}{\mathcal{V}}

\newcommand*{\prs}{\PR\text{\,--\,\,}a.s.}
\newcommand*{\pts}{\dif\PR\!\times\!\dif t\text{\,--\,\,}a.e.}

\newcommand*{\dte}{\dif t\text{\,--\,\,}a.e.}
\newcommand{\tT}{[0,T]}
\newcommand{\intT}[2][T]{\int^{#1}_{#2}}
\newcommand{\me}{\mathrm{e}}
\newcommand{\one}[1]{{\bf 1}_{#1}}
\newcommand{\vp}{\varepsilon}
\newcommand{\EXlr}[1]{\mathrm{\mathbf{E}}\left[#1\right]}
\DeclareMathOperator*{\esssup}{esssup}
\DeclareMathOperator*{\essinf}{essinf}

\usepackage{cleveref}
\crefname{thm}{Theorem}{Theorems}
\crefname{pro}{Proposition}{Propositions}
\crefname{lem}{Lemma}{Lemmas}
\crefname{rmk}{Remark}{Remarks}
\crefname{cor}{Corollary}{Corollaries}
\crefname{dfn}{Definition}{Definitions}
\crefname{ex}{Example}{Examples}
\crefname{section}{Section}{Sections}
\crefname{subsection}{Subsection}{Subsections}

\newtheorem{thm}{Theorem}
\newtheorem{lem}[thm]{Lemma}
\newtheorem{pro}[thm]{Proposition}
\newtheorem{rmk}[thm]{Remark}
\newtheorem{dfn}[thm]{Definition}

\journal{}
\allowdisplaybreaks[4]

\begin{document}
	
\begin{frontmatter}

\title{{\boldmath\bf Stochastic differential games with state constraints and Isaacs equations with nonlinear Neumann problems}\tnoteref{found}}
\tnotetext[found]{Supported by the National Natural Science Foundation of China (Nos.\;11371362 and 11601509) and the Natural Science Foundation of Jiangsu Province (No.\;BK20150167).}
\author{Lishun Xiao\corref{cor1}}
\ead{xiaolishun@cumt.edu.cn}
\cortext[cor1]{Corresponding author}
\author{Dejian Tian}%

\address{School of Mathematics, China University of Mining and Technology, Xuzhou, Jiangsu, 221116, P.R. China}

\begin{abstract}
  We investigate a two-player zero-sum stochastic differential game problem with the state process being constrained in a connected bounded closed domain, and the cost functional described by the solution of a generalized backward stochastic differential equation (GBSDE for short). We show that the value functions enjoy a (strong) dynamic programming principle, and are the unique viscosity solution of the associated Hamilton-Jacobi-Bellman-Isaacs equations with nonlinear Neumann boundary problems. To obtain the existence for viscosity solutions, we provide a new approach utilizing the representation theorem for generators of the GBSDE, which is proved by a random time change method and is a novel result in its own right.
\end{abstract}

\end{frontmatter}

\section{Introduction}
\label{sec:Introduction}

This paper is concerned with a two-player zero-sum stochastic differential game problem with state constraints and recursive cost functionals. The presence of state constraints refers to the requirement that the state process lives in a connected bounded closed  domain, where the convexity is unnecessary in our framework. And the cost functional is a recursive one because it is governed by the solution of a certain backward stochastic differential equations (BSDEs for abbreviation). The main objective of this paper is to establish a (strong) dynamic programming principle (DPP for short) for this control problem and characterize the value function as a unique viscosity solution of associated Hamilton-Jacobi-Bellman-Isaacs equations with nonlinear Neumann boundary problems.

\citet*{FlemingSouganidis1989IUMJ} initially explored the two-player zero-sum stochastic differential game, which translated from former purely deterministic differential game, such as \citet*{EvansSouganidis1984IUMJ}, into the stochastic framework. Based on their works, many researchers developed the stochastic differential game to different directions, see \citet*{Swiech1996JMAA},  \citet*{BuckdahnCardaliaguetRainer2004SICON}, \citet*{BayraktarPoor2005SICON} and the references therein. Based on the pioneer works of \citet*{PardouxPeng1990SCL}, \citet*{Peng1992SSR} introduced the nonlinear BSDE theory to the stochastic recursive optimal control and obtained the probabilistic interpretation for associated HJB equations. Their dynamics consists of a controlled coupled forward-backward stochastic differential equation, and the forward one describes the state process and the backward one induces the cost functional. Some researchers also studied zero-sum stochastic differential games governed by BSDEs, such as \citet*{HamadeneLepeltier1995SSR,HamadeneLepeltier1995SCL}, but they need the diffusion coefficient is non-degenerate and independent of controls.  \citet*{BuckdahnLi2008SICON} eliminated these restrictions and improved corresponding results of \citet*{FlemingSouganidis1989IUMJ} with two main differences. The first one is that their admissible control processes can depend on the full past of the trajectories of the driving Brownian motion; the second one is that their cost functional is induced by a controlled BSDE. 

The state constraints for stochastic differential game arise naturally in many practical applications. A primary motivation is considered in the pursuit-evasion game model. For instance, the pursuer and evader move in a prescribed region, and the cost functional is the capture time, see \citet*{CardaliaguetQuincampoixSaint-Pierre2001SICON} for a survey. Recently, \citet*{Krylov2014PTRF} studied a stochastic differential game with state constraints using first exit time from a domain. Nevertheless, the recursive case of stochastic differential game with state constraints has not been widely studied, especially the probabilistic interpretation, in viscosity sense, for Isaacs equations with nonlinear Neumann boundary problems. Most recently, \citet*{LiTang2015ESAIMCOCV} and \citet*{BiswasIshiiSahaWang2017SICON} investigated probabilistic interpretation for nonlinear Neumann problems of HJB equations under different types. The former adopted the optimal control of recursive type but the latter did not. To prove the probabilistic interpretation, \citet*{LiTang2015ESAIMCOCV} employed Peng's approximation method proposed in \citet*{Peng1997SuiJiFenXiXuanJiang}. This method is extensively distributed in different frameworks, see \citet*{BuckdahnLi2008SICON}, \citet*{LiWei2014SICON}, and \citet*{BuckdahnNie2016SICON} for a survey. In this paper we will introduce a new approach to prove the probabilistic interpretation for Isaacs equations utilizing the representation theorem for generators of BSDEs. This representation theorem is originally proved by \citet*{BriandCoquetHuMeminPeng2000ECP} and then further extended by \citet*{Jiang2008AAP}. 

The dynamics of our stochastic differential game with state constraints is given by the following controlled reflected stochastic differential equation (RSDE for short), 
\begin{equation*}
  \begin{cases}
  \displaystyle X^{t,x;u,v}_s=x+\!\!\int^s_t\!\!b(r,X^{t,x;u,v}_r,u_r,v_r)\dif s+\!\int^s_t\!\!\sigma(r,X^{t,x;u,v}_r,u_r,v_r)\dif B_r+\!\!\int^s_t\!\nabla\phi(X^{t,x;u,v}_r)\dif\eta^{t,x;u,v}_r,\\[7pt]
  \displaystyle  \eta^{t,x;u,v}_s=\int^s_t\one{\partial\calO}(X^{t,x;u,v}_r)\dif\eta^{t,x;u,v}_r,\quad \eta^{t,x;u,v}_\cdot\text{ is increasing},\quad s\in[t,T],
  \end{cases}
\end{equation*}
where $(t,x)\in\tT\times\overline \calO$ is the initial data, $u(\cdot)$ and $v(\cdot)$ are two admissible controls, $\eta^{t,x;u,v}_\cdot$ is an adapted continuous increasing process, which is the reflecting process that keeps $X^{t,x;u,v}_\cdot$ from leaving the connected bounded closed domain $\overline \calO$. And the recursive cost functional is defined by $J(t,x;u,v)=Y^{t,x;u,v}_t$, $(t,x)\in\tT\times\overline \calO$, where $Y^{t,x;u,v}_\cdot$ is the unique solution of the following Generalized BSDE (GBSDE for short) introduced by \citet*{PardouxZhang1998PTRF}, setting $\Theta:=(X,Y,Z)$,
\begin{equation}\label{eq:GBSDEInIntroduction}
  \begin{cases}
  -\dif Y^{t,x;u,v}_s=g(s,\Theta^{t,x;u,v}_s,u_s,v_s)\dif s+f(s,X^{t,x;u,v}_s,Y^{t,x;u,v}_s,u_s,v_s)\dif \eta^{t,x;u,v}_s-\langle Z^{t,x;u,v}_s,\dif B_s\rangle,\\
  Y^{t,x;u,v}_T=\Phi(X^{t,x;u,v}_T),\quad s\in[t,T],
  \end{cases}
\end{equation}
where $(X^{t,x;u,v}_\cdot,\eta^{t,x;u,v}_\cdot)$ is the unique solution of RSDE. As usual, the lower and upper value functions of our stochastic differential game with state constraints are, respectively, defined as follows,
\begin{equation*}
  W(t,x):=\essinf_{\beta\in\calB_{t,T}}\esssup_{u\in\calU_{t,T}}J(t,x;u,\beta[u]),\quad
  U(t,x):=\esssup_{\alpha\in\calA_{t,T}}\essinf_{v\in\calV_{t,T}}J(t,x;\alpha[v],v),
\end{equation*}
where $\alpha$ and $\beta$ are, respectively, strategies for player I and II.

In this paper we aim to construct a strong DPP for the lower value function $W(t,x)$, in which the intermediate time is a random time instead of the deterministic one. Using this DPP we illustrate that $W(t,x)$ is the unique viscosity solution of Isaacs equations with nonlinear Neumann problems,
\begin{equation}\label{eq:IsaacsEquationWithNeumannIntroduction}
  \begin{cases}
  \displaystyle 
  \partial_tW(t,x)+H^-(t,x,W,\nabla W,D^2W)=0,&(t,x)\in[0,T)\times\calO,\\
  \displaystyle 
  \frac{\partial}{\partial n}W(t,x)+\sup_{u\in U}\inf_{v\in V}f(t,x,W(t,x),u,v)=0,& (t,x)\in[0,T)\times\partial\calO,\\
  \displaystyle 
  W(T,x)=\Phi(x),& x\in\overline{\calO},
  \end{cases}
\end{equation}
with the Hamiltonian defined as $H^-(t,x,y,p,A):=\sup_{u\in U}\inf_{v\in V}H(t,x,y,p,A,u,v)$, where $H$ is defined as follows,
\begin{equation*}
  H(t,x,y,p,A,u,v)
  :=\frac{1}{2}Tr\{\sigma\sigma^*(t,x,u,v)A\}+\langle b(t,x,u,v),p\rangle+g(t,x,y,\sigma^*p,u,v).
\end{equation*}
Similarly, the upper value function $U(t,x)$ enjoys the symmetric features. 

Another remarkable result of this paper is the representation theorem for generators of GBSDEs. For brevity, we write the GBSDE as the following form with slight abuse of notions, for given $(t,y,z)\in\tT\times\rtn\times\rtn^d$ and $0<\vp\leq T-t$,
\begin{equation}\label{eq:GBSDEWithoutControlsInIntroduction}
  Y^\vp_s=y+\langle z,B_{\tau_{t+\vp}}-B_t\rangle +\int^{\tau_{t+\vp}}_sg(r,Y^\vp_r,Z^\vp_r)\dif r +\int^{\tau_{t+\vp}}_sf(r,Y^\vp_r)\dif A_r-\int^{\tau_{t+\vp}}_s\langle Z^\vp_r,\dif B_r\rangle,
\end{equation}
where the functions $g:\Omega\times\tT\times\rtn\times\rtn^d\mapsto\rtn$ and $f:\Omega\times\tT\times\rtn\mapsto\rtn$ are called generators, $A_\cdot$ is a given adapted continuous increasing process and $\tau_\cdot$ is the inverse function of $A_s+s-A_t$, $s\in[t,T]$. The representation theorem for generators of GBSDE \eqref{eq:GBSDEWithoutControlsInIntroduction} we obtained can be roughly interpreted as that there exist a pair of positive processes $(a_\cdot,b_\cdot)$ with $a_\cdot+b_\cdot=1$ such that
\[a_tg(t,y,z)+b_tf(t,y)=\lim_{\vp\to0^+}\frac{1}{\vp}(Y^\vp_t-y),\quad \prs.\]

Special emphasis should be given to the approach we adopted to prove that $W(t,x)$ is a viscosity solution of the Isaacs equation \eqref{eq:IsaacsEquationWithNeumannIntroduction}. In our approach the representation theorem for generators of GBSDEs plays the essential role. Compared with Peng's approximation method, the representation theorem approach is more straightforward and applicable to general frameworks (such as, non-Lipschitz settings). Moreover, our approach can be easily extended into uncontrolled and unconstrained cases, i.e., the probabilistic interpretation for viscosity solution of semilinear and quasilinear PDEs. This can also be regarded as a new application of the representation theorem.

We would also like to mention that the usual method to such representation ceases to work for the case of GBSDEs, such as \citet*{BriandCoquetHuMeminPeng2000ECP} and \citet*{Jiang2008AAP}. The usual method to the classical representation theorem ($f\equiv 0$ or $A_\cdot\equiv 0$) relies heavily on the Lebesgue Lemma (see Proposition 2.2 in \citet*{Jiang2008AAP}), which can be stated briefly as that $\lim_{\vp\to0^+}\int^{t+\vp}_tg_s\dif s=g_t$ holds for almost every $t$. However, this Lebesgue Lemma is not applicable to the random measure $\dif A_r$, i.e., $\lim_{\vp\to0^+}\int^{t+\vp}_tf_r\dif A_r\neq f_t$. This brings a great difficulty to represent both of $g(t,y,z)$ and $f(t,y)$, even one of them. Here we adopt a method of random time change to address this issue. Applying a time change we can transform the random measure $\dif A_r$ to a Lebesgue measure $\dif r$, combine the terms $g(r,y,z)$ and $f(r,y)$ into a new generator $a_rg(\tau_r,y,z)+b_rf(\tau_r,y)$ and transform the Brownian motions to special martingales, whence the GBSDE is transformed to a BSDE driven by martingales. So the representation problem for generators of GBSDEs is transformed to the counterpart of BSDEs driven by martingales. Some fine properties for these special martingales inherit from standard Brownian motions, which insures the representation theorem for generators of BSDEs driven by martingales.  

The paper is organized as follows. \cref{sec:Preliminaries} gives necessary notations and some elementary results about GBSDEs. \cref{sec:RepresentationTheoremForGBSDEs} introduces the representation theorem for generators of GBSDEs by the method of random time change.  \cref{sec:SDGWithStateConstraintsAndDPP} demonstrates the formulation of stochastic differential game with state constraints, the (strong) DPP and regularity property for the lower and upper value functions. \cref{sec:ViscositySolutionOfHJBI} shows that the lower and upper value functions are the unique viscosity solution of associated Isaacs equations with nonlinear Neumann boundary problem. Finally, some complementary results are provided in \cref{sec:AppendixComplementaryResults}, including some extended regularity of solutions of RSDEs and GBSDEs with respect to initial data.

\section{Preliminaries}\label{sec:Preliminaries}
In this paper, $T>0$ is a given real number, $(\Omega,\calF,\PR)$ is a classical Wiener space, and the driving Brownian motion $B$ will be the coordinate process on $\Omega$. Precisely, $\Omega$ will denote the set of continuous functions from $\tT$ to $\rtn^d$, i.e., $\Omega:=C^0(\tT;\rtn^d)$; $\calF$ is the Borel $\sigma$-algebra over $\Omega$, completed with respect to the Wiener measure $\PR$, and $B$ denotes the coordinate process $B_t(\omega)=\omega_t$, $t\in\tT$, $\omega\in\Omega$. Let $(\calF_t)_{t\geq 0}$ be the natural $\sigma$-algebra filtration generated by $(B_t)_{t\geq 0}$ and augmented by all $\PR$-null sets. We denote by $\calT_{\tau_1,\tau_2}$ the set of all $(\calF_t)$-stopping times with values in $[\tau_1,\tau_2]$.

The Euclidean norms of a vector $x\in\rtn^n$ and a matrix $z\in\rtn^{n\times d}$  will be denoted by $|x|$ and $|z|:=\sqrt{Tr(zz^*)}$, where and hereafter $z^*$ represents the transpose of $z$. We denote by $\calS^2(0,T;\rtn^n)$ the set of $\rtn^n$-valued, $(\calF_t)$-adapted and continuous processes $(y_t)_{t\in\tT}$ such that $\EX[\sup_{t\in\tT}|y_t|^2]<\infty$. Let $\calH^2(0,T;\rtn^n)$ denote the set of $\rtn^n$-valued and $(\calF_t)$-progressively measurable processes $(z_t)_{t\in\tT}$ satisfying that $\EX[\intT{0}|z_s|^2\dif s]<\infty$.
Moreover, let $\calA(0,T;\rtn)$ represent the set of all real-valued, continuous increasing and $(\calF_t)$-progressively measurable processes whose paths vanish at $t=0$.

Next we introduce a generalized BSDE (GBSDE for short) of the following type:
\begin{equation}\label{eq:GBSDEOneDimensional}
  Y_t=\xi+\int^T_tg(s,Y_s,Z_s)\dif s+\int^T_tf(s,Y_s)\dif A_s-\int^T_t\langle Z_s,\dif B_s\rangle,\quad t\in\tT,
\end{equation}
where $A_\cdot\in\calA(0,T;\rtn)$, $\xi\in L^2(\Omega,\calF_T,\PR;\rtn)$ such that $\EX[\me^{\lambda A_T}|\xi|^2]<\infty$ for all $\lambda>0$,  $g:\Omega\times\tT\times\rtn\times\rtn^d\mapsto\rtn$ and $f:\Omega\times\tT \times\rtn\mapsto\rtn$ are called generators, and $g(\cdot,\cdot,y,z)$ and $f(\cdot,\cdot,y)$ are both $(\calF_t)$-progressively measurable for each $(y,z)\in\rtn\times\rtn^d$. A GBSDE associated with prescribed parameters $(\xi,T,g,f,A)$ is also denoted by GBSDE $(\xi,T,g+f\dif A)$. We call $(Y_t,Z_t)_{t\in\tT}$ a solution of GBSDE \eqref{eq:GBSDEOneDimensional} if it belongs to $\calS^2(0,T;\rtn)\times\calH^2(0,T;\rtn^d)$ and satisfies this GBSDE almost surely. The existence and uniqueness for solutions is given by \citet*{PardouxZhang1998PTRF} under the following assumptions.
\begin{enumerate}
	\renewcommand{\theenumi}{(A\arabic{enumi})}
	\renewcommand{\labelenumi}{\theenumi}
	\item\label{A:GFContinuousInY} $\pts$, $g(t,\cdot,z)$ and $f(t,\cdot)$ are continuous for all $z\in\rtn^d$;
	\item\label{A:GFConditionsInYZ} There exist some constants $\lambda_1$, $\lambda_2\in\rtn$ and $K\geq 0$, and two adapted processes $(g_t,f_t)_{t\in\tT}$ valued in $[1,\infty)$ such that $\pts$, for each $y$, $y_1$, $y_2\in\rtn$ and $z$, $z_1$, $z_2\in\rtn^d$,
	\begin{enumerate}
		\renewcommand{\theenumii}{(\roman{enumii})}
		\renewcommand{\labelenumii}{\theenumii}
		\item\label{A:ItemGFMonotonicInY} $( y_1-y_2)\big(g(t,y_1,z)-g(t,y_2,z)\big)\leq \lambda_1|y_1-y_2|^2,\ (y_1-y_2)\big( f(t,y_1)-f(t,y_2)\big)\leq\lambda_2|y_1-y_2|^2$;
		\item\label{A:ItemGLipschitzInZ} $|g(t,y,z_1)-g(t,y,z_2)|\leq K|z_1-z_2|$;
		\item\label{A:ItemGFLinearInY} $|g(t,y,z)|\leq g_t+K(|y|+|z|),\quad |f(t,y)|\leq f_t+K|y|$;
		\item\label{A:ItemGFSquareIntegrability} $\EX\big[\int^T_0\me^{\lambda A_t}|g_t|^2\dif t\big]+\EX\big[\int^T_0\me^{\lambda A_t}|f_t|^2\dif A_t\big]<\infty$ for all $\lambda>0$.
	\end{enumerate}
\end{enumerate}

\begin{lem}\label{lem:EstimateForGBSDEYZDrivenByA}
	Let $A_\cdot\in\calA(0,T;\rtn)$, $\xi\in L^2(\Omega,\calF_T,\PR;\rtn)$, \ref{A:GFContinuousInY} and \ref{A:GFConditionsInYZ} hold, and $(Y_t,Z_t)_{t\in\tT}$ is a solution of GBSDE \eqref{eq:GBSDEOneDimensional}. Then for any $\lambda>0$ there exists a constant $C\geq0$ depending on $\lambda_1$, $\lambda_2$, $K$ and $T$ such that
	\begin{align*}
	  &\EX\bigg[\sup_{t\in\tT}\me^{\lambda A_t}|Y_t|^2+\int^T_0\me^{\lambda A_t}|Y_t|^2\dif A_t+\int^T_0\me^{\lambda A_t}|Z_t|^2\dif t\bigg]\\
	  &\leq C\EX\bigg[\me^{\lambda A_T}|\xi|^2+\int^T_0\me^{\lambda A_t}|g(t,0,0)|^2\dif t+\int^T_0\me^{\lambda A_t}|f(t,0)|^2\dif A_t\bigg].
	\end{align*}
\end{lem}

\begin{lem}\label{lem:AprioriEstimateForDifferenceOfYs}
	For each $i=1$, $2$, $A^i_\cdot\in\calA(0,T;\rtn)$, $\xi^i\in L^2(\Omega,\calF_T,\PR;\rtn)$, $g^i$ and $f^i$ satisfy \ref{A:GFContinuousInY} and \ref{A:GFConditionsInYZ}, $(Y^i_t,Z^i_t)_{t\in\tT}$ is a solution of GBSDE $(\xi^i,T,g^i+f^i\dif A^i)$. Denote $\hat\Psi:=\Psi^1-\Psi^2$ with $\Psi=\xi$, $A$, $g$, $f$, $Y_\cdot$ and $Z_\cdot$. Then for any $\lambda>0$ there exists a constant $C\geq 0$ such that
	\begin{align*}
	  \EX\bigg[\sup_{t\in\tT}\me^{\lambda k_t}|\hat Y_t|^2+\int^T_0\me^{\lambda k_t}|\hat Z_t|^2\dif t\bigg]
	  \leq{}&C\EX\bigg[\me^{\lambda k_T}|\hat \xi|^2+\int^T_0\me^{\lambda k_t}|\hat g(t,Y^1_t,Z^1_t)|^2\dif t\\
	  &+\int^T_0\me^{\lambda k_t}|\hat f(t,Y^1_t)|^2\dif A^2_t+\int^T_0\me^{\lambda k_t}|f^1(t,Y^1_t)|^2\dif |\hat A|_t\bigg],
	\end{align*} 
	where $k_t:=|\hat A|_t+A^2_t$ with $|\hat A|_t$ denoting the total variation of the process $\hat A_\cdot$ in the interval $[0,t]$.
\end{lem}
\begin{lem}\label{lem:ComparisonTheoremForSolutionsOfGBSDEs}
	Let assumptions of \cref{lem:AprioriEstimateForDifferenceOfYs} are in force, and $A^1_\cdot=A^2_\cdot$, $\xi^1\leq\xi^2$, $g^1(t,y,z)\leq g^2(t,y,z)$ and $f^1(t,y)\leq f^2(t,y)$ for all $(y,z)\in\rtn\times\rtn^d$, $\pts$. Then $\prs$, $Y^1_t\leq Y^2_t$ for all $t\in\tT$. 
\end{lem}

\section{Representation theorem for generators of GBSDEs}\label{sec:RepresentationTheoremForGBSDEs}
In this section we want to show the representation theorem for generators of GBSDE \eqref{eq:GBSDEOneDimensional} by a method of random time change. As aforementioned in the Introduction, the method of random time change avoid the great difficulty brought from the random measure $\dif A_\cdot$. Our representation theorem is a nontrivial extension of the classical one, including \citet*{BriandCoquetHuMeminPeng2000ECP} and \citet*{Jiang2008AAP}. 

\subsection{The method of random time change}
In this subsection we will introduce the method of random time change. Before that we need the following lemma, which interprets a necessary and sufficient condition for absolute continuity. The first assertion in \cref{lem:ZareskyLemma} was obtained by M. A. Zaretsky (see, e.g., Ex 3.21 in \citet*{Leoni2009Book}), and the second one is an immediate consequence of the inverse function theorem. 
\begin{lem}\label{lem:ZareskyLemma}
	Let $h:[a,b]\mapsto\rtn$ be continuous and strictly increasing, then its inverse function $H:[h(a),h(b)]\mapsto\rtn$ is absolutely continuous if and only if the set $\{x\in [a,b]: \nabla h(x)=0\}$ has Lebesgue measure zero. Moreover, if $|\nabla h(x)|\geq\vp>0$ almost everywhere, then $\nabla H(x)\leq 1/\vp$ almost everywhere.
\end{lem}

Next we will give some notions and well-known results about random time change, which can be find in \cite[Chapter 2]{IkedaWatanabe1989Book} and \cite[Chapter V]{RevuzYor2005Book}. 
\begin{dfn}\label{dfn:TimeChangeProcess}
	By a process of time change $\psi$ we mean any continuous $(\calF_t)$-adapted process $(\psi_t)_{t\geq 0}$ such that $\psi_0=0$, $\prs$, $t\mapsto \psi_t$ is strictly increasing and $\lim_{t\to\infty}\psi_t=\infty$. 
\end{dfn}
For a given process of time change $\psi$ and $t\in[0,\infty)$, we define, with the convention $\inf\{\varnothing\}=\infty$, 
\[\tau_s:=\inf\{t\geq 0:\psi_t>s\}.\]
We have the following assertions. $\tau_0=0$, $\tau_\cdot$ is continuous and strictly increasing and $\lim_{s\to\infty}\tau_s=\infty$. So the process $\tau$ is called a time change associated with $\psi$. Furthermore, $\tau_\cdot$ coincides with the inverse function of $\psi_\cdot$ in pathwise sense. The family $(\tau_s)_{s\geq 0}$ is a family of $(\calF_s)$-stopping times, and for every $t$, the random variable $\psi_t$ is a $(\calF_{\tau_s})$-stopping time. We also have that $\psi_t=\inf\{s\geq 0:\tau_s>t\}$, which indicates that $\psi_\cdot$ coincides with the inverse function of $\tau_\cdot$ in pathwise sense. Then $\psi_{\tau_t}=\tau_{\psi_t}=t$. 

Set $\widetilde\calF_t:=\calF_{\tau_t}$. Thus $(\widetilde{\calF}_t)_{t\geq 0}$ is a reference family on $(\Omega,\calF,\PR)$ and satisfies the usual conditions. Let $(X_t)_{t\geq 0}$ be a $(\calF_t)$-progressively measurable process and define $\widetilde X_t:=X_{\tau_t}$. Then $(\widetilde X_t)_{t\geq 0}$ is a $(\widetilde\calF_t)$-progressively measurable process. We call $X_{\tau_\cdot}$ the time change of $X$ by $\tau_\cdot$, and denote $T_\tau X:=X_\tau$. 

We denote $\calM^{c,loc}_{2}$ the family of all continuous locally square integrable martingales $(M_t)_{t\geq 0}$ relative to $(\calF_t)$ such that $M_0=0$, and $\widetilde\calM^{c,loc}_{2}$ the similar family relative to $\widetilde\calF_t$. Then the map $T_\tau(\cdot):\calM^{c,loc}_{2}\mapsto\widetilde\calM^{c,loc}_{2}$ is a bijection which preserves all structures in the space of $\calM^{c,loc}_{2}$. Also, the random time change commutes with Lebesgue-Stieltjes integrals and stochastic integrals, see the following lemma.
\begin{lem}\label{lem:TimeChangeFormula}
	Suppose that $\tau$ is a time change associated with $\psi$, $(H_t)_{t\geq 0}$ is $(\calF_t)$-progressively measurable process. We have
	\begin{enumerate}
		\item[(i)] If $(A_t)_{t\geq 0}$ is a continuous process of finite variation, then 
		\[\int^{\tau_s}_{0}H_r\dif A_r=\int^s_0 H_{\tau_r}\dif A_{\tau_r},\quad 0\leq s<\infty;\]
		\item[(ii)] If $(M_t)_{t\geq 0}$ belongs to $\calM^{c,loc}_2$ and $(H_t)_{t\geq 0}$ satisfies $\prs$, $\int^\infty_0H^2_t\dif \langle M\rangle_t<\infty$, then
		\[\int^{\tau_s}_0 H_r\dif M_r=\int^s_0 H_{\tau_r}\dif M_{\tau_r},\quad 0\leq s<\infty.\]
	\end{enumerate}
\end{lem}

We proceed to present a crucial \cref{lem:TimeChangeToAbsoluteContinuity} which illustrates our method of time change.  The first assertion in \cref{lem:TimeChangeToAbsoluteContinuity} is an immediate consequence of \cref{lem:ZareskyLemma} and the Radon-Nikodym theorem, and it can be interpreted as that any continuous and increasing process can be transformed by a random time change to an absolutely continuous process; the second one gives some fine properties of time changed Brownian motions, see Proposition V.1.5 and Theorem V.1.6 in \citet*{RevuzYor2005Book} for more details.
\begin{lem}\label{lem:TimeChangeToAbsoluteContinuity}
	For each $(A_t)_{t\in\tT}\in\calA(0,T;\rtn)$, define $\psi:\Omega\times\tT\mapsto\rtn^+$ as $\psi_t:=A_t+t$. Then 
	\begin{enumerate}
		\item[(i)] the inverse function $\tau_\cdot$ of $\psi_\cdot$ exists with $\dif \tau_s=a_s\dif s$ and $\dif A_{\tau_s}=b_s\dif s$, where $(a_s,b_s)_{s\in[0,\psi_T]}$ is a pair of real-valued positive $(\calF_{\tau_s})$-progressively measurable processes with $a_\cdot+b_\cdot=1$ and $a_\cdot>0$;
		\item[(ii)] $(B_{\tau_s})_{s\geq 0}$ is a continuous $(\calF_{\tau_s})$-martingale, and enjoys the properties that for each $s\geq 0$, $\EX[B_{\tau_s}]=0$, $\EX[|B_{\tau_s}|^2]=\EX[d\tau_s]$ and the quadratic variation of the $i$th coordinate $\langle B^i_\tau\rangle_s=\langle B^i\rangle_{\tau_s}=\tau_s$, $\prs$.
	\end{enumerate}
\end{lem}

We denote by $\calH^2_M(0,T;\rtn^d)$ (or $\calH^2_M$) the set of $(\calF_t)$-progressively measurable $\rtn^d$-valued processes $(z_s)_{s\in\tT}$ such that $\|z\|^2_{\calH^2_M}:=\EX[\int^T_0|z_s|^2 \dif\langle M\rangle_t]<\infty$, where $M_\cdot\in \calM^{c,loc}_2$. Then $\calH^2_M$ is a Hilbert space. With the help of \cref{lem:TimeChangeFormula,lem:TimeChangeToAbsoluteContinuity}, it is straightforward to verify the following \cref{pro:EquiventBSDEAndGBSDEOfSolutions}, in which the same notations in \cref{lem:TimeChangeToAbsoluteContinuity} are adopted. This proposition interprets that GBSDE \eqref{eq:GBSDEOneDimensional} is equivalent to a stochastic interval BSDE driven by a continuous martingale. The novelty lies in that the $f(s,Y_s)\dif A_s$ is transformed to a part of the new generator of the BSDE driven by martingale. 
\begin{pro}\label{pro:EquiventBSDEAndGBSDEOfSolutions}
	Let $\xi\in L^2(\Omega,\calF_T,\PR;\rtn)$, $(A_t)_{t\in\tT}\in\calA(0,T;\rtn)$. If GBSDE \eqref{eq:GBSDEOneDimensional}, which is duplicated as follows,
	\begin{equation*}
	Y_t=\xi+\int^T_tg(s,Y_s,Z_s)\dif s+\int^T_tf(s,Y_s)\dif A_s-\int^T_t\langle Z_s,\dif B_s\rangle, \quad t\in\tT.
	\end{equation*} 
	admits a unique solution $(Y_t,Z_t)_{t\in\tT}$ in $\calS^2\times\calH^2$ with filtration $(\calF_t)_{t\geq 0}$, then the following BSDE 
	\begin{equation}\label{eq:TimeChangedGBSDEEquivBSDE}
	\widetilde Y_t=\xi+\int^{\psi_T}_t\big(a_sg(\tau_s,\widetilde Y_s,\widetilde Z_s)+b_sf(\tau_s,\widetilde Y_s)\big)\dif s-\int^{\psi_T}_t\langle\widetilde Z_s,\dif B_{\tau_s}\rangle,\quad t\in[0,\psi_T].
	\end{equation}	
	admits a unique solution $(\widetilde Y_s:=Y_{\tau_s},\widetilde Z_s:=Z_{\tau_s})_{s\in[0,\psi_T]}$ in $\calS^2\times\calH^2_{B_\tau}$ with filtration $(\calF_{\tau_s})_{s\geq 0}$. Conversely, if BSDE \eqref{eq:TimeChangedGBSDEEquivBSDE} admits a unique solution $(\widetilde Y_s,\widetilde Z_s)_{s\in[0,\psi_T]}$ in $\calS^2\times\calH^2_{B_\tau}$ with filtration $(\calF_{\tau_s})_{s\geq 0}$, then $(Y_t:=\widetilde Y_{\psi_t},Z_t:=\widetilde Z_{\psi_t})_{t\in[0,T]}$ in $\calS^2\times\calH^2$ with filtration $(\calF_t)_{t\geq 0}$ is the unique solution of GBSDE \eqref{eq:GBSDEOneDimensional}.
\end{pro}

\subsection{Representation theorem for generators}

Fix a triplet $(t,y,z)\in[0,T)\times\rtn\times\rtn^d$, and choose a $\vp>0$ with $\vp\leq T-t$. For a given $(A_t)_{t\in\tT}\in\calA(0,T;\rtn)$, define $\psi:\Omega\times[t,T]\mapsto\rtn^+$ as 
\begin{equation}\label{eq:DfnPsiInRepresentationTheorem}
\psi_s:=A_s+s-A_t.
\end{equation}
It is easily to see that $\psi_\cdot$ enjoys all the properties in \cref{dfn:TimeChangeProcess}. Then the corresponding conclusions in \cref{lem:TimeChangeToAbsoluteContinuity} applies. We denote the inverse function of $\psi_\cdot$ by $\tau_\cdot$ and utilize the same notations of \cref{lem:TimeChangeToAbsoluteContinuity}. Under assumptions \ref{A:GFContinuousInY} -- \ref{A:GFConditionsInYZ}, GBSDE \eqref{eq:GBSDEOneDimensional} admits a unique solution. Since $\tau_\cdot$ belongs to $\calT_{t,T}$, similar arguments to Theorem 12 in \citet*{XiaoFan2017Stochastics} yield that the following GBSDE,
\begin{equation}\label{eq:GBSDERepresentationDTPlusVp}
Y_s=y+\langle z,B_{\tau_{t+\vp}}-B_t\rangle+\int^{\tau_{t+\vp}}_sg(r,Y_r,Z_r)\dif r-\int^{\tau_{t+\vp}}_sf(r,Y_r)\dif A_r-\int^{\tau_{t+\vp}}_s\langle Z_r,\dif B_r\rangle,\quad s\in[t,\tau_{t+\vp}],
\end{equation}
admits a unique solution in $\calS^2\times\calH^2$ with filtration $(\calF_t)_{t\geq 0}$,
\[\big(Y_s(y+\langle z,B_{\tau_{t+\vp}}-B_t\rangle,\tau_{t+\vp},g+f\dif A),Z_s(y+\langle z,B_{\tau_{t+\vp}}-B_t\rangle,\tau_{t+\vp},g+f\dif A)\big)_{s\in[t,\tau_{t+\vp}]},\]
which is also denoted by $(Y^\vp_s,Z^\vp_s)_{s\in[t,\tau_{t+\vp}]}$.  Then we have the following representation theorem for generator of GBSDEs.
\begin{thm}[Representation Theorem]
	\label{thm:RepresentationTheoremForGenerator}
	For given $(A_t)_{t\in\tT}\in\calA(0,T;\rtn)$, assume that \ref{A:GFContinuousInY} -- \ref{A:GFConditionsInYZ} hold, define $\psi_\cdot$ as in \eqref{eq:DfnPsiInRepresentationTheorem} and denote by its inverse function $\tau_\cdot$. Then there exist a pair of real-valued positive $(\calF_{\tau_s})$-progressively measurable processes $a_\cdot$ and $b_\cdot$ with $a_\cdot+b_\cdot=1$ and $a_\cdot>0$ such that for each $y\in\rtn$, $z\in\rtn^d$, $1\leq p<2$ and $\dte$ $t\in[0,T)$,
	\begin{gather}
	L^p-\lim_{\vp\to 0^+}\frac{1}{\vp}
	\bigg((Y^\vp_t-y)-\EX\bigg[\int^{t+\vp}_t[a_rg(\tau_r,y,z)-b_rf(\tau_r,y)]\dif r\bigg|\calF_{\tau_t}\bigg]\bigg)=0,\label{eq:RepThmForGBSDEIntegralType}\\
	a_tg(t,y,z)+b_tf(t,y)=L^p-\lim_{\vp\to 0^+}\frac{1}{\vp}
	\left[
	Y^\vp_t
	-y
	\right].\label{eq:RepThmForGBSDEIdentityType}
	\end{gather}
\end{thm}

\begin{proof}
	Let all the assumptions hold. Fix a triplet $(t,y,z)\in[0,T)\times\rtn\times\rtn^d$ and choose a $\vp>0$ with $\vp\leq T-t$. To simply notations, we denote the solution of GBSDE \eqref{eq:GBSDERepresentationDTPlusVp} by  $(Y^\vp_s,Z^\vp_s)_{s\in[t,\tau_{t+\vp}]}$. We make a time change to all the processes involved in GBSDE \eqref{eq:GBSDERepresentationDTPlusVp} by setting $\widetilde\Psi_\cdot:=\Psi_{\tau_\cdot}$ with $\Psi=Y^\vp$, $Z^\vp$, $B$. Then \cref{pro:EquiventBSDEAndGBSDEOfSolutions} indicates that 
	\begin{equation*}
	\widetilde Y^\vp_s=y+\langle z,\widetilde B_{t+\vp}-\widetilde B_t\rangle+\int^{t+\vp}_s\big(a_rg(\tau_r,\widetilde Y^\vp_r,\widetilde Z^\vp_r)+b_rf(\tau_r,Y^\vp_r)\big)\dif r-\int^{t+\vp}_s\langle \widetilde Z^\vp_r,\dif\widetilde B_r\rangle,\quad s\in[t,t+\vp],
	\end{equation*}
	where $\dif \tau_r=a_r\dif r$, $\dif A_{\tau_r}=b_r\dif r$ and $a_\cdot+b_\cdot=1$ with $a_\cdot>0$. We set $\overline Y^\vp_\cdot:=\widetilde Y^\vp_\cdot-y-\langle z,\widetilde B_\cdot-\widetilde B_t\rangle$ and $\overline Z^\vp_\cdot:=\widetilde Z^\vp_\cdot-z$, then for each $s\in[t,t+\vp]$,
	\begin{equation}\label{eq:BSDEOverlineYZDrivenByBDr}
	\overline Y^\vp_s=\int^{t+\vp}_s\overline g(r,\overline Y^\vp_r,\overline Z^\vp_r)\dif r-\int^{t+\vp}_s\langle \overline Z^\vp_r,\dif\widetilde B_r\rangle,
	\end{equation}
	where for each $y'\in\rtn$ and $z'\in\rtn^d$, we write
	\[\overline g(r,y',z'):=a_rg(\tau_r,y'+y+\langle z,\widetilde B_r-B_t\rangle,z'+z)+b_rf(\tau_r,y'+y+\langle z,\widetilde B_r-B_t\rangle).\] 
	It is evident that \ref{A:GFContinuousInY} and \ref{A:ItemGLipschitzInZ} are fulfilled by $\overline g$, i.e., $\overline g(r,y',z')$ is continuous in $y'$ and Lipschitz continuous in $z'$, the Lipschitz constant is $Ka_r$. Moreover, it follows from \ref{A:ItemGFMonotonicInY} that $\prs$, for each $y_1$, $y_2\in\rtn$ and $z'\in\rtn^d$,
	\begin{equation}\label{eq:OverlineGMonotonicityCondition}
	\langle y_1-y_2,\overline g(r,y_1,z')-\overline g(r,y_2,z')\rangle
	\leq \lambda_1 a_r|y_1-y_2|^2+\lambda_2 b_r|y_1-y_2|^2,
	\end{equation}
	which means that \ref{A:ItemGFMonotonicInY} holds for $\overline g$, i.e., the monotonicity condition holds for $\overline g$. Next, by the linear growth of $g$ and $f$ in \ref{A:ItemGFLinearInY}, we deduce that for each $y'\in\rtn$ and $z'\in\rtn^d$,
	\begin{align*}
	|\overline g(r,y',z')|\leq a_r|g_{\tau_r}|+b_r|f_{\tau_r}|+2K(|y'|+|y|)+2K|z||\widetilde B_r-B_t|+K|z'|+K|z|.
	\end{align*}
	Then (i) in \cref{lem:TimeChangeToAbsoluteContinuity} and \ref{A:ItemGFLinearInY} yield that 
	\begin{align*}
	\EX\bigg[\int^{t+\vp}_ta_r|g_{\tau_r}|^2\dif r\bigg]=\EX\bigg[\int^{t+\vp}_t|g_{\tau_r}|^2\dif \tau_r\bigg]=\EX\bigg[\int^{\tau_{t+\vp}}_t|g_r|^2\dif r\bigg]\leq \EX\bigg[\int^T_0|g_r|^2\dif r\bigg]<\infty.
	\end{align*}
	With analogous arguments we can obtain that $\EX[\int^{t+\vp}_tb_r|f_{\tau_r}|^2\dif r]<\infty$. Moreover, it follows from the fact $\tau_\cdot\leq T$ and (ii) in \cref{lem:TimeChangeToAbsoluteContinuity} that
	\begin{equation*}
	\EXlr{\int^{t+\vp}_t|\widetilde B_r-B_t|^2\dif r}\leq 2\int^{t+\vp}_t\EXlr{|\widetilde B_r|^2}\dif r+2t\vp\leq 2\vp(dT+t)<\infty.
	\end{equation*}
	Thereby, we know that \ref{A:ItemGFLinearInY} also holds for $\overline g$. 
	
	Next we will build an estimate for solutions of BSDE \eqref{eq:BSDEOverlineYZDrivenByBDr}. For a constant $\lambda\geq 0$ which will be chosen later, It\^o's formula to $\me^{\lambda r}|\overline Y^\vp|^2$ yields that, for each $s\in[t,t+\vp]$,
	\begin{align*}
	&\me^{\lambda s}|\overline Y^\vp_s|^2+\lambda\int^{t+\vp}_s\me^{\lambda r}|\overline Y_r^\vp|^2\dif r+\int^{t+\vp}_s\me^{\lambda r}|\overline Z^\vp_r|^2\dif \tau_r\\
	&\leq 2\int^{t+\vp}_s\me^{\lambda r}\langle \overline Y^\vp_r,\overline g(r,\overline Y^\vp_r,\overline Z^\vp_r)\rangle \dif r-2\int^{t+\vp}_s\me^{\lambda r}\langle \overline Y^\vp_r,\overline Z^\vp_r\dif \widetilde B_r\rangle.
	\end{align*}
	The inner product including $\overline g$ can be enlarged by \eqref{eq:OverlineGMonotonicityCondition} as follows, 
	\begin{align*}
	2\langle \overline Y^\vp_r,\overline g(r,\overline Y^\vp_r,\overline Z^\vp_r)\rangle
	&\leq 2\lambda_1a_r|\overline Y^\vp_r|^2+2\lambda_2b_r|\overline Y^\vp_r|^2+2a_rK|\overline Y^\vp_r||\overline Z^\vp_r|+2|\overline Y^\vp_r||\overline g(r,0,0)|\\
	&\leq 2(\lambda_1+\lambda_2+K^2)|\overline Y^\vp_r|^2+\frac{a_r}{2}|\overline Z^\vp_r|^2+2|\overline Y^\vp_r||\overline g(r,0,0)|.
	\end{align*}
	Thus, by choosing $\lambda\geq 2(\lambda_1+\lambda_2+K^2)$ and the fact $\dif \tau_r=a_r\dif r$, we deduce that for each $s\in[t,t+\vp]$,
	\begin{align*}
	\EXlr{\left.\int^{t+\vp}_s\me^{\lambda r}|\overline Z^\vp_r|^2\dif \tau_r\right|\calF_{\tau_s}}
	\leq 2\EXlr{\left.\int^{t+\vp}_s\me^{\lambda r}|\overline Y^\vp_r||\overline g(r,0,0)|\dif r\right|\calF_{\tau_s}}.
	\end{align*}
	Moreover, we have that 
	\begin{align*}
	\sup_{r\in[t,t+\vp]}\me^{\lambda r}|\overline Y^\vp_r|^2\leq 2\int^{t+\vp}_t\me^{\lambda r}|\overline Y^\vp_r||\overline g(r,0,0)|\dif r+2\sup_{s\in[t,t+\vp]}\left|\int^{s}_t\me^{\lambda r}\langle \overline Y^\vp_r,\overline Z^\vp_r\dif \widetilde B_r\rangle\right|.
	\end{align*}
	Then it follows from Burkholder-Davis-Gundy's inequality and the basic inequality $2ab\leq 2a^2+b^2/2$ that there exists a generic constant $C\geq 0$, which will change from line to line, such that  
	\begin{align*}
	\EX\bigg[\sup_{r\in[t,t+\vp]}\me^{\lambda r}|\overline Y^\vp_r|^2\bigg|\calF_{\tau_t}\bigg]\!
	&\leq2\EX\bigg[\!\int^{t+\vp}_t\!\me^{\lambda r}|\overline Y^\vp_r||\overline g(r,0,0)|\dif r\bigg|\calF_{\tau_t}\bigg]\!
	+C\EX\bigg[\Big(\int^{t+\vp}_t\!\me^{2\lambda r}|\overline Y^\vp_r|^2|\overline Z^\vp_r|^2\dif \tau_r\Big)^{1/2}\bigg|\calF_{\tau_t}\bigg]\\
	&\hspace{-3.5cm}\leq 2\EX\bigg[\!\int^{t+\vp}_t\!\me^{\lambda r}|\overline Y^\vp_r||\overline g(r,0,0)|\dif r\bigg|\calF_{\tau_t}\bigg]
	+C^2\EX\bigg[\int^{t+\vp}_t\me^{\lambda r}|\overline Z^\vp_r|^2\dif \tau_r\bigg|\calF_{\tau_t}\bigg]+\frac{1}{4}\EX\bigg[\sup_{s\in[t,t+\vp]}\me^{\lambda s}|\overline Y^\vp_s|^2\bigg|\calF_{\tau_t}\bigg].
	\end{align*}
	Immediately, we obtain that 
	\begin{align*}
	\EX\bigg[\sup_{r\in[t,t+\vp]}\me^{\lambda r}|\overline Y^\vp_r|^2\bigg|\calF_{\tau_t}\bigg]+\EX\bigg[\int^{t+\vp}_t\me^{\lambda r}|\overline Z^\vp_r|^2\dif \tau_r\bigg|\calF_{\tau_t}\bigg]
	\leq C\EX\bigg[\int^{t+\vp}_t\me^{\lambda r}|\overline Y^\vp_r||\overline g(r,0,0)|\dif r\bigg|\calF_{\tau_t}\bigg].
	\end{align*}   
	Finally, the right hand side term of the previous inequality can be estimated as follows, 
	\begin{align*}
	C\EX\bigg[\!\int^{t+\vp}_t\!\!\me^{\lambda r}|\overline Y^\vp_r||\overline g(r,0,0)|\dif r\bigg|\calF_{\tau_t}\!\bigg]\!
	\leq\! \frac{1}{4}\EX\bigg[\sup_{s\in[t,t+\vp]}\me^{\lambda s}|\overline Y^\vp_s|^2\bigg|\calF_{\tau_t}\!\bigg]\!
	+\!C^2\EX\bigg[\Big(\!\int^{t+\vp}_t\!\!\me^{\lambda r/2}|\overline g(r,0,0)|\!\dif r\Big)^2\bigg|\calF_{\tau_t}\!\bigg],
	\end{align*}
	which indicates the following estimate for solutions of BSDE \eqref{eq:BSDEOverlineYZDrivenByBDr},
	\begin{align*}
	\EX\bigg[\sup_{r\in[t,t+\vp]}\me^{\lambda r}|\overline Y^\vp_r|^2\bigg|\calF_{\tau_t}\bigg]+\EX\bigg[\int^{t+\vp}_t\me^{\lambda r}|\overline Z^\vp_r|^2\dif \tau_r\bigg|\calF_{\tau_t}\bigg]
	\leq C\EX\bigg[\Big(\int^{t+\vp}_t\me^{\lambda r/2}|\overline g(r,0,0)|\dif r\Big)^2\bigg|\calF_{\tau_t}\bigg].
	\end{align*}
	Furthermore, the previous estimate and H\"older's inequality yields that 
	\begin{equation*}
	\frac{1}{\vp}\EX\bigg[\sup_{r\in[t,t+\vp]}|\overline Y^\vp_r|^2+\int^{t+\vp}_t|\overline Z^\vp_r|^2\dif \tau_r\bigg]
	\leq C\EX\bigg[\int^{t+\vp}_t|\overline g(r,0,0)|^2\dif r\bigg].
	\end{equation*}
	Thus, the absolute continuity of integrals indicates that
	\begin{equation}\label{eq:LimitVpToZeroOverlineYZTendToZero}
	\lim_{\vp\to0^+}\frac{1}{\vp}\EX\bigg[\sup_{r\in[t,t+\vp]}|\overline Y^\vp_r|^2+\int^{t+\vp}_t|\overline Z^\vp_r|^2\dif \tau_r\bigg]=0.
	\end{equation}
	
	Next taking $s=t$ and then conditional expectation with respect to $\calF_{\tau_t}$ in both sides of BSDE \eqref{eq:BSDEOverlineYZDrivenByBDr} lead to the following identity, $\prs$,
	\begin{equation*}
	\frac{1}{\vp}(\widetilde Y^\vp_t-y)=\frac{1}{\vp}\overline Y^\vp_t=\frac{1}{\vp}\EX\bigg[\int^{t+\vp}_t\overline g(r,\overline Y^\vp_r,\overline Z^\vp_r)\dif r\bigg|\calF_{\tau_t}\bigg].
	\end{equation*}
	We set
	\begin{equation*}
	M^\vp_t:=\frac{1}{\vp}\EX\bigg[\int^{t+\vp}_t\overline g(r,\overline Y^\vp_r,\overline Z^\vp_r)\dif r\bigg|\calF_{\tau_t}\bigg],\quad N^\vp_t:=\frac{1}{\vp}\EX\bigg[\int^{t+\vp}_t\overline g(r,0,0)\dif r\bigg|\calF_{\tau_t}\bigg].
	\end{equation*}
	Hence, it holds that
	\begin{equation*}
	\frac{1}{\vp}(\widetilde Y^\vp_t-y)-a_tg(t,y,z)-b_tf(t,y)=\frac{1}{\vp}\overline Y^\vp_t-\overline g(t,0,0)=\frac{1}{\vp}Y^\vp_t-\overline g(t,0,0)=M^\vp_t-N^\vp_t+N^\vp_t-\overline g(t,0,0).
	\end{equation*}
	Then it reduces to prove that $(M^\vp_t-N^\vp_t)$ and $(N^\vp_t-\overline g(t,0,0))$ tend to $0$ in $L^p$ $(1\leq p<2)$ sense for $\dte$ $t\in[0,T)$ as $\vp\to 0^+$, respectively. To this end, we should employ the following proposition, which is a corollary of Proposition 2 in \citet*{FanJiangXu2011EJP}.  
	\begin{pro}\label{pro:ApproximationForOverlineG}
		Assume that the generator $\overline g$ satisfies \ref{A:GFContinuousInY} and \ref{A:ItemGFLinearInY}, and let $y\in\rtn$. Then there exist a nonnegative $(\calF_{\tau_t})$-progressively measurable process sequence $\{(\overline g^n_t)_{t\in\tT}\}^\infty_{n=1}$ depending on $y$ such that $\lim_{n\to 0}\EX[|\overline g^n_t|^2]=0$ for $\dte$ $t\in\tT$, and $\prs$, for each $n\geq 1$ and $y'\in\rtn$, 
		\begin{align*}
		|\overline g(t,y',0)-\overline g(t,y,0)|&\leq 2n|y-y'|+\overline g^n_t.
		\end{align*}
	\end{pro}
	Continue the proof of \cref{thm:RepresentationTheoremForGenerator}. It follows from Jesen's and H\"older's inequalities, \ref{A:ItemGLipschitzInZ} for $\overline g$ and \cref{pro:ApproximationForOverlineG} that for $\dte$ $t\in[0,T)$, $n\geq 1$ and $1\leq p<2$,
	\begin{align}
	\EXlr{|M^\vp_t-N^\vp_t|^p}
	&\leq \EX\bigg[\bigg(\frac{1}{\vp}\int^{t+\vp}_t|\overline g(r,\overline Y^\vp_r,\overline Z^\vp_r)-\overline g(r,0,0)|\dif r\bigg)^p\bigg]\nonumber\\
	&\leq 2^p\EX\bigg[\bigg(\frac{1}{\vp}\int^{t+\vp}_t2n|\overline Y^\vp_r|+Ka_r|\overline Z^\vp_r|\dif r\bigg)^p\bigg]+2^p\EX\bigg[\bigg(\frac{1}{\vp}\int^{t+\vp}_t|\overline g^n_r|\dif r\bigg)^p\bigg]\nonumber\\
	&\leq 4^p(2n+K)^p\EX\bigg[\frac{1}{\vp}\int^{t+\vp}_t(|\overline Y^\vp_r|^p+a_r|\overline Z^\vp_r|^p)\dif r\bigg]+2^p\EX\bigg[\bigg(\frac{1}{\vp}\int^{t+\vp}_t|\overline g^n_r|\dif r\bigg)^p\bigg].\label{eq:MMinusNLeqYZAndG}
	\end{align}
	The first term on the right hand side of the previous inequality tends to $0$ as $\vp\to0^+$ because of \eqref{eq:LimitVpToZeroOverlineYZTendToZero}. Concerning the second term, Proposition 2.2 in \citet*{Jiang2008AAP} implies that for $\dte$ $t\in[0,T)$, each $n\geq 1$ and $1\leq p<2$,
	\begin{equation*}
	\lim_{\vp\to0^+}\EX\bigg[\bigg(\frac{1}{\vp}\int^{t+\vp}_t|\overline g^n_r|\dif r\bigg)^p\bigg]=\EX[|\overline g^n_t|^p].
	\end{equation*}
	Note that the right hand side in the previous identity tends to $0$ as $n\to\infty$. Hence, by sending $\vp\to0^+$ and then $n\to\infty$ in \eqref{eq:MMinusNLeqYZAndG}, we get that for $\dte$ $t\in[0,T)$, $\lim_{\vp\to0^+}\EX[|M^\vp-N^\vp|^p]=0$. Then we get the identity \eqref{eq:RepThmForGBSDEIntegralType}. 
	
	Now we consider the term $(N^\vp_t-\overline g(t,0,0))$. It follows from Jensen's inequality and Proposition 2.2 in \citet*{Jiang2008AAP} that for $\dte$ $t\in[0,T)$, and each $1\leq p<2$, as $\vp\to0^+$,
	\begin{equation*}
	\EXlr{|N^\vp_t-\overline g(t,0,0)|^p}\leq \EX\bigg[\bigg(\frac{1}{\vp}\int^{t+\vp}_t|\overline g(r,0,0)-\overline g(t,0,0)|\dif r\bigg)^p\bigg]\to 0.
	\end{equation*}
	Hence, we get the identity \eqref{eq:RepThmForGBSDEIdentityType}. Then the proof of \cref{thm:RepresentationTheoremForGenerator} is completed.
\end{proof}

\section{Stochastic differential games with state constraints and dynamic programming principle}
\label{sec:SDGWithStateConstraintsAndDPP}
In this section we will show the dynamic programming principle (DPP) for the stochastic differential game with state being constrained in a connected bounded closed domain, where the state equation is induced by a controlled RSDE and the cost functional is given by a GBSDE. We clarify that the control state space $U$ (resp., $V$) is a compact metric space, and the admissible control set $\calU$ (resp., $\calV$) for the player I (resp., II) is the set of all $U$ (resp., $V$)-valued $(\calF_t)$-progressively measurable processes. 

Let $\calO$ be an open connected bounded subset of $\rtn^n$ given by $\calO=\{x\in\rtn^n:\phi(x)>0\}$ with $\phi\in C^2(\rtn^n;\rtn)$, and such that $\partial\calO=\{x\in\rtn^n:\phi(x)=0\}$, with $|\nabla\phi(x)|=1$ for all $x\in\partial\calO$. Observe that $\nabla\phi(x)$ coincides with the unit normal pointing toward the interior of $\calO$ at $x\in\partial\calO$. Another observation is that $\phi$, $\nabla\phi$ and $D^2\phi$ are bounded in $\overline{\calO}$. Then there exists a constant $C_0>0$ such that
\begin{equation}\label{eq:InequalityOfBoundedSetO1}
2\langle x'-x,\nabla\phi(x)\rangle+C_0|x-x'|^2\geq 0, \quad \forall\; x\in\partial\calO,\; x'\in\overline\calO.
\end{equation}
We also postulate that $0\in\calO$. For given admissible controls $u(\cdot)\in\calU$ and $v(\cdot)\in\calV$, the corresponding state processes starting from $\zeta\in L^2(\Omega,\calF_t,\PR;\overline \calO)$ at the initial time $t\in\tT$ is governed by the following RSDE,
\begin{equation}\label{eq:RSDEControled}
  \begin{cases}
    \displaystyle X^{t,\zeta;u,v}_s=\zeta\!+\!\!\int^s_t\!\!b(r,X^{t,\zeta;u,v}_r,u_r,v_r)\dif s+\!\int^s_t\!\!\sigma(r,X^{t,\zeta;u,v}_r,u_r,v_r)\dif B_r+\!\!\int^s_t\!\nabla\phi(X^{t,\zeta;u,v}_r)\dif\eta^{t,\zeta;u,v}_r,\\[7pt]
    \displaystyle  \eta^{t,\zeta;u,v}_s=\int^s_t\one{\partial\calO}(X^{t,\zeta;u,v}_r)\dif\eta^{t,\zeta;u,v}_r,\quad \eta^{t,\zeta;u,v}_\cdot\text{ is increasing},\quad s\in[t,T].
  \end{cases}
\end{equation}
Here, the deterministic functions $b:\tT\times\overline \calO\times U\times V\mapsto\rtn^n$ and $\sigma:\tT\times\overline \calO\times U\times V\mapsto\rtn^{n\times d}$  satisfy the following assumptions:
\begin{enumerate}
	\renewcommand{\theenumi}{(H\arabic{enumi})}
	\renewcommand{\labelenumi}{\theenumi}
	\item\label{H:BSigmaBoundedWithControls} $b$ and $\sigma$ are uniformly bounded, and for each $x\in\rtn^n$, $b(\cdot,x,\cdot,\cdot)$, $\sigma(\cdot,x,\cdot,\cdot)$ are continuous;
	\item\label{H:BSigmaLipschitzInXWithControls} There exists a constant $K\geq 0$ such that for all $t\in\tT$, each $x_1$, $x_2\in\overline \calO$ and $(u,v)\in U\times V$, 
	\[|b(t,x_1,u,v)-b(t,x_2,u,v)|+|\sigma(t,x_1,u,v)-\sigma(t,x_2,u,v)|\leq K|x_1-x_2|.\]
\end{enumerate}
Then by Theorem 1 in \citet*{Marin-RubioReal2004JTP} we know that RSDE \eqref{eq:RSDEControled} admits a unique solution $(X^{t,\zeta;u,v}_s,\eta^{t,\zeta;u,v}_s)_{s\in[t,T]}$, which is $(\calF_s)$-progressively measurable and values in $\overline \calO\times\rtn^+$. \cref{pro:EstimateForRSDEHatXEta} in \cref{sec:AppendixComplementaryResults} indicates the following estimates. For each $t\in\tT$, $0\leq \delta\leq T-t$, $\zeta$, $\zeta'\in L^4(\Omega,\calF_t,\PR;\overline \calO)$, $u(\cdot)\in\calU$ and $v(\cdot)\in\calV$, there exists a constant $C\geq 0$ depending on $K$, $T$, $\phi$, $b$ and $\sigma$ such that
\begin{equation*}
 	\EXlr{\left.\sup_{s\in[t,T]}\big|X^{t,\zeta;u,v}_s-X^{t,\zeta';u,v}_s\big|^4+\sup_{s\in[t,T]}\big|\eta^{t,\zeta;u,v}_s-\eta^{t,\zeta';u,v}_s\big|^4\right|\calF_t}\leq C|\zeta-\zeta'|^4;
\end{equation*}
\begin{equation}\label{eq:EstimateForXEtaDominatedByZetaWithControls}
	\EXlr{\left.\sup_{s\in[t,T]}\big|X^{t,\zeta;u,v}_s\big|^4+\sup_{s\in[t,T]}\big|\eta^{t,\zeta;u,v}_s\big|^4\right|\calF_t}\leq C(1+|\zeta|^4);
\end{equation}
\begin{equation}\label{eq:EstimateForXInitialEtaWithTimeDelayWithControls}
  \EX\left[\left.\sup_{s\in[t,t+\delta]}|X^{t,\zeta;u,v}_s-\zeta|^4\right|\calF_t\right]\leq C\delta^2;\quad \EX\left[\left.|\eta^{t,\zeta;u,v}_{t+\delta}|^4\right|\calF_t\right]\leq C\delta^2.
\end{equation}

For simplicity of notations, we will denote $M:=\overline \calO\times\rtn\times\rtn^{d}$; an element in $M$ is denoted by $\Theta:=(X,Y,Z)$ with $X\in\overline \calO$, $Y\in\rtn$ and $Z\in\rtn^{d}$. Similarly, we use $\theta:=(x,y,z)$, and so on. Next we introduce the following controlled GBSDE, for given admissible controls $u(\cdot)\in\calU$ and $v(\cdot)\in\calV$, 
\begin{equation}\label{eq:GBSDEControlled}
\begin{cases}
  -\dif Y^{t,\zeta;u,v}_s=g(s,\Theta^{t,\zeta;u,v}_s,u_s,v_s)\dif s+f(s,X^{t,\zeta;u,v}_s,Y^{t,\zeta;u,v}_s,u_s,v_s)\dif \eta^{t,\zeta;u,v}_s-\langle Z^{t,\zeta;u,v}_s,\dif B_s\rangle,\\
  Y^{t,\zeta;u,v}_T=\Phi(X^{t,\zeta;u,v}_T),\quad s\in[t,T],
\end{cases}
\end{equation}
where $(X^{t,\zeta;u,v}_s,\eta^{t,\zeta;u,v}_s)_{s\in[t,T]}$ is the unique solution of RSDE \eqref{eq:RSDEControled}, the mappings $\Phi:\rtn^n\mapsto\rtn$, $g:\tT\times M\times U\times V$, $f\mapsto\tT\times\overline \calO\times\rtn\times U\times V$ satisfy the following assumptions:
\begin{enumerate}
	\renewcommand{\theenumi}{(H\arabic{enumi})}
	\renewcommand{\labelenumi}{\theenumi}
	\setcounter{enumi}{2}
	\item\label{H:GFContinuousWithControls} For each $x\in\rtn^n$ and $z\in\rtn^d$, $g(\cdot,x,\cdot,z,\cdot,\cdot)$ and $f(t,x,y,\cdot,\cdot)$ are continuous, and  $f(\cdot,\cdot,\cdot,u,v)\in C^{1,2,2}(\tT\times\rtn^n\times\rtn;\rtn)$; 
	\item\label{H:GFMonotonicInYLipschitzInZWithControls} There exist some constants $\lambda_1$, $\lambda_2\in\rtn$ and $K\geq 0$ such that for each $t\in\tT$, $\theta$, $\theta_1$, $\theta_2\in M$, and $(u,v)\in U\times V$,
	\begin{enumerate}
		\renewcommand{\theenumii}{(\roman{enumii})}
		\renewcommand{\labelenumii}{\theenumii}
		\item\label{H:ItemGFMonotonicInYWithControls} $(y_1-y_2)\big(g(t,x,y_1,z,u,v)-g(t,x,y_2,z,u,v)\big)\leq \lambda_1|y_1-y_2|^2$, 
		
		$(y_1-y_2)\big( f(t,x,y_1,u,v)-f(t,x,y_2,u,v)\big)\leq\lambda_2|y_1-y_2|^2$;
		\item\label{H:ItemPhiGFLipschitzInZWithControls}$|\Phi(x_1)-\Phi(x_2)|+|g(t,x_1,y,z_1,u,v)-g(t,x_2,y,z_2,u,v)|+|f(t,x_1,y,u,v)-f(t,x_2,y,u,v)|\leq K(|x_1-x_2|+|z_1-z_2|)$,
		\item\label{H:ItemGFLinearGrowthInYWithControls} $|g(t,0,y,0,u,v)|+|f(t,0,y,u,v)|\leq K(1+|y|)$.
	\end{enumerate}
\end{enumerate}
It is evident that the coefficient $\Phi$ satisfies the global linear growth condition in $x$, i.e., $|\Phi(x)|\leq C(1+|x|)$. \cref{pro:EstimateForRSDEHatXEta}, \ref{H:ItemPhiGFLipschitzInZWithControls} and \ref{H:ItemGFLinearGrowthInYWithControls} yield that \ref{A:ItemGFLinearInY} and \ref{A:ItemGFSquareIntegrability} hold, then  GBSDE \eqref{eq:GBSDEControlled} admits a unique solution $(Y^{t,\zeta;u,v}_s,Z^{t,\zeta;u,v}_s)_{s\in[t,T]}\in\calS^2(t,T;\rtn)\times\calH^2(t,T;\rtn^d)$. \cref{pro:EstimatesForSolutionsOfGBSDEsAppendix} in \cref{sec:AppendixComplementaryResults} implies that there exists some constant $C\geq 0$ such that for all $t\in\tT$, $\zeta$, $\zeta'\in L^2(\Omega,\calF_t,\PR;\overline \calO)$, $u(\cdot)\in\calU$ and $v(\cdot)\in\calV$, $\prs$,
\begin{equation}\label{eq:ValueFunctionalJHolderContinuousInX}
  |Y^{t,\zeta;u,v}_t-Y^{t,\zeta';u,v}_t|\leq C(|\zeta-\zeta'|+|\zeta-\zeta'|^{1/2}),\quad |Y^{t,\zeta;u,v}_t|\leq C(1+|\zeta|).
\end{equation}

Next we define some subspaces of admissible controls and the admissible strategies for the game, which are borrowed from \citet*{BuckdahnLi2008SICON}. 
\begin{dfn}
	An admissible control process $(u_r)_{r\in[t,s]}$ (resp., $(v_r)_{r\in[t,s]}$) for player I (resp., II) on $[t,s]$ $(t<s\leq T)$ is a $(\calF_r)$-progressively measurable process taking values in $U$ (resp., $V$). The set of all admissible controls for player I (resp.,  II) on $[t,s]$ is denoted by $\calU_{t,s}$ (resp., $\calV_{t,s}$). We identify two processes $u$ and $\overline u$ in $\calU_{t,s}$ and write $u\equiv \overline u$ on $[t,s]$, if $\PR\{u=\overline u,\; a.e. \text{ in } [t,s]\}=1$. Similarly, we interpret $v\equiv\overline v$ on $[t,s]$ in $\calV_{t,s}$.
\end{dfn}
\begin{dfn}
	A non-anticipative strategy for player I on $[t,s]$ $(t<s\leq T)$ is a mapping $\alpha:\calV_{t,s}\mapsto\calU_{t,s}$ such that, for any stopping time $S\in\calT_{t,s}$ and any $v_1$, $v_2\in\calV_{t,s}$, with $v_1\equiv v_2$ on $[t,S]$, it holds that $\alpha[v_1]\equiv\alpha[v_2]$ on $[t,S]$. Non-anticipative strategies for player II on $[t,s]$, $\beta:\calU_{t,s}\mapsto\calV_{t,s}$, are define similarly. The set of all non-anticipative strategies $\alpha:\calV_{t,s}\mapsto\calU_{t,s}$ for player I on $[t,s]$ is denoted by $\calA_{t,s}$. The set of all non-anticipative strategies $\beta:\calU_{t,s}\mapsto\calV_{t,s}$ for player II on $[t,s]$ is denoted by $\calB_{t,s}$. 
\end{dfn}

Given the admissible control processes $u(\cdot)\in\calU_{t,T}$ and $v(\cdot)\in\calV_{t,T}$, we define the  associated cost functional as follows:
\begin{equation*}
  J(t,x;u,v):=Y^{t,x;u,v}_s|_{s=t},\quad (t,x)\in\tT\times\overline\calO,
\end{equation*}
where the process $Y^{t,x;u,v}_\cdot$ is the uniqueness solution of controlled GBSDE \eqref{eq:GBSDEControlled}. Since $J(t,x;u,v)$ is continuous in $x$, by some approximation arguments we can get that $\prs$,  $J(t,\zeta;u,v)=Y^{t,\zeta;u,v}_t$ holds for each $t\in\tT$ and $\zeta\in L^4(\Omega,\calF_t,\PR;\overline \calO)$.

Let us now define the lower and upper value functions of our stochastic differential game with state constraints as follows, respectively,
\begin{align}\label{eq:DefOfValueFunctionWAndU}
  W(t,x):=\essinf_{\beta\in\calB_{t,T}}\esssup_{u\in\calU_{t,T}}J(t,x;u,\beta[u]);\ 
  U(t,x):=\esssup_{\alpha\in\calA_{t,T}}\essinf_{v\in\calV_{t,T}}J(t,x;\alpha[v],v),\  (t,x)\in\tT\times\overline \calO.
\end{align}
We can prove that the previous lower and upper value functions are deterministic, see \cref{pro:LowerUpperValueFunctionDeterministic}. The Girsanov transformation method and the uniqueness for solutions of RSDE \eqref{eq:RSDEControled} and GBSDE \eqref{eq:GBSDEControlled} play a key role in the proof. The arguments are very analogous to Proposition 3.3 in \citet*{BuckdahnLi2008SICON} so we omit them.   
\begin{pro}\label{pro:LowerUpperValueFunctionDeterministic}
	For each $(t,x)\in\tT\times\overline \calO$, we have $W(t,x)=\EX[W(t,x)]$ and $U(t,x)=\EX[U(t,x)]$, $\prs$. By identifying $W(t,x)$ and $U(t,x)$ with their deterministic versions, we can consider $W$, $U:\tT\times\overline \calO\mapsto\rtn$ as  deterministic functions.
\end{pro}

Next we focus on the study of the properties of $W(t,x)$ because the counterparts for $U(t,x)$ can be obtained similarly. As a consequence of \eqref{eq:ValueFunctionalJHolderContinuousInX} and \eqref{eq:DefOfValueFunctionWAndU}, $W(t,x)$ is continuous in $x$.
\begin{thm}\label{thm:LowerValueFunctionWHolderContinuousInX}
	There exists a constant $C>0$ such that for all $t\in\tT$, $x$ and $x'\in\overline \calO$, 
	\begin{gather*}
	  |W(t,x)-W(t,x')|\leq C(|x-x'|+|x-x'|^{1/2}),\\
	  |W(t,x)|\leq C(1+|x|).
	\end{gather*}
\end{thm}

Before illustrating the (weak) DPP for the lower value function $W(t,x)$, we should adapt the notion of backward semigroups, initiated by \citet*{Peng1997SuiJiFenXiXuanJiang}, from the BSDE case to the GBSDE case. For each given initial data $(t,x)\in\tT\times\overline \calO$, any two stopping times $\tau\in\calT_{t,T}$ and $\sigma\in\calT_{t,\tau}$, two admissible control processes $u(\cdot)\in\calU_{t,\tau}$ and $v(\cdot)\in\calV_{t,\tau}$ and a random variable $\xi\in L^2(\Omega,\calF_{\tau},\PR;\rtn)$, we put 
\begin{equation*}
  G^{t,x;u,v}_{\sigma,\tau}[\xi]:=\bar Y^{t,x;u,v}_\sigma,
\end{equation*}
where the pair $(\bar Y^{t,x;u,v}_s,\bar Z^{t,x;u,v}_s)_{s\in[t,\tau]}$ is the solution of the following GBSDE, 
\begin{equation*}
  \begin{cases}
  -\dif\bar Y^{t,x;u,v}_s=g(s,X^{t,x;u,v}_s,\bar Y^{t,x;u,v}_s,\bar Z^{t,x;u,v}_s,u_s,v_s)\dif s -\langle \bar Z^{t,x;u,v}_s,\dif B_s\rangle\\
  \hspace{2.2cm}+f(s,X^{t,x;u,v}_s,\bar Y^{t,x;u,v}_s,u_s,v_s)\dif \eta^{t,x;u,v}_s,\\
  \bar Y^{t,x;u,v}_\tau=\xi,\quad s\in[t,\tau],
  \end{cases}
\end{equation*}
and $(X^{t,x;u,v}_\cdot,\eta^{t,x;u,v}_\cdot)$ is the solution of RSDE \eqref{eq:RSDEControled}. Then, concerning the solution $Y^{t,x;u,v}_\cdot$ of GBSDE \eqref{eq:GBSDEControlled} we have the flow property for the backward semigroup $G$, i.e., for each $\tau\in\calT_{t,T}$,
\begin{equation*}
  J(t,x;u,v)=Y^{t,x;u,v}_t=G^{t,x;u,v}_{t,T}[\Phi(X^{t,x;u,v}_T)]=G^{t,x;u,v}_{t,\tau}[Y^{t,x;u,v}_\tau].
\end{equation*}

Now we present the corresponding (weak) DPP. Its proof is quite analogous to that of Theorem 3.6 in \citet*{BuckdahnLi2008SICON} because all the major tools employed by them hold true in our framework, such as the non-anticipativity property of $\beta$, the uniqueness for solutions of RSDE \eqref{eq:RSDEControled} and GBSDE \eqref{eq:GBSDEControlled} and \cref{thm:LowerValueFunctionWHolderContinuousInX}. Thus, we omit its proof.
\begin{thm}[Weak DPP]\label{thm:DPPWeakVersion}
	Assume that \ref{H:BSigmaBoundedWithControls} -- \ref{H:GFMonotonicInYLipschitzInZWithControls} hold. For each $(t,x)\in\tT\times\overline \calO$ and $0<\delta\leq T-t$, the lower value function $W(t,x)$ enjoys the following DPP:
	\begin{equation*}
	  W(t,x)=\essinf_{\beta\in\calB_{t,t+\delta}}\esssup_{u\in\calU_{t,t+\delta}}G^{t,x;u,\beta[u]}_{t,t+\delta}[W(t+\delta,X^{t,x;u,\beta[u]}_{t+\delta})].
	\end{equation*}
\end{thm}

\begin{rmk}\label{rmk:InequalitiesFromWeakDPP}
	It is easily followed from the previous DPP (or Remark 3.4 in \citet*{BuckdahnLi2008SICON}) that for each $(t,x)\in\tT\times\overline \calO$, $0<\delta\leq T-t$ and $\vp>0$, 
	\begin{enumerate}
		\item[(i)] for each $\beta[\cdot]\in\calB_{t,t+\delta}$, there exists some $u^\vp(\cdot)\in\calU_{t,t+\delta}$ such that 
		\begin{equation*}
		  W(t,x)\leq G^{t,x;u^\vp,\beta[u^\vp]}_{t,t+\delta}[W(t+\delta,X^{t,x;u^\vp,\beta[u^\vp]}_{t+\delta})]+\vp,\quad \prs;
		\end{equation*}
		\item[(ii)] there exists some $\beta^\vp[\cdot]\in\calB_{t,t+\delta}$ such that for all $u(\cdot)\in\calU_{t,t+\delta}$,
		\begin{equation*}
		  W(t,x)\geq G^{t,x;u,\beta^\vp[u]}_{t,t+\delta}[W(t+\delta,X^{t,x;u,\beta^\vp[u]}_{t+\delta})]-\vp,\quad \prs.
		\end{equation*}
	\end{enumerate}
\end{rmk}

With the help of \cref{rmk:InequalitiesFromWeakDPP}, we can prove the continuity of the lower value function $W(t,x)$ with respect to $t$.
\begin{thm}\label{thm:LowerValueFunctionWContinuousInT}
	Assume that \ref{H:BSigmaBoundedWithControls} -- \ref{H:GFMonotonicInYLipschitzInZWithControls} hold. Then for each $x\in\overline \calO$ and $t$, $t'\in\tT$, there exists a constant $C\geq 0$ such that
	\begin{equation*}
	  |W(t,x)-W(t',x)|\leq C(|t-t'|^{1/2}+|t-t'|^{1/4}).
	\end{equation*}
\end{thm}
\begin{proof}
	Let assumptions hold, $(t,x)\in\tT\times\overline \calO$ and $0<\delta\leq T-t$. It is sufficient to prove the following inequality by \cref{rmk:InequalitiesFromWeakDPP},
	\begin{equation*}
	  -C(\delta^{1/2}+\delta^{1/4})\leq W(t,x)-W(t+\delta,x)\leq C(\delta^{1/2}+\delta^{1/4}).
	\end{equation*}
	We only prove the second inequality in the previous inequality since the other one can be shown in a similar way. It follows from (i) in \cref{rmk:InequalitiesFromWeakDPP} that for small enough $\vp>0$, arbitrarily chosen $\beta[\cdot]\in\calB_{t,t+\delta}$ and $u^\vp(\cdot)\in\calU_{t,t+\delta}$,
	\begin{equation*}
	  W(t,x)-W(t+\delta,x)\leq I^1_\delta+I^2_\delta+\vp,
	\end{equation*}
	where 
	\begin{align*}
	  I^1_\delta&:=G^{t,x;u^\vp,\beta[u^\vp]}_{t,t+\delta}[W(t+\delta,X^{t,x;u^\vp,\beta[u^\vp]}_{t+\delta})]-G^{t,x;u^\vp,\beta[u^\vp]}_{t,t+\delta}[W(t+\delta,x)],\\
	  I^2_\delta&:=G^{t,x;u^\vp,\beta[u^\vp]}_{t,t+\delta}[W(t+\delta,x)]-W(t+\delta,x).
	\end{align*}
	We now estimate $I^1_\delta$ and $I^2_\delta$ respectively. In view of the notion of backward semigroups, \cref{lem:AprioriEstimateForDifferenceOfYs}, \cref{thm:LowerValueFunctionWHolderContinuousInX} and \eqref{eq:EstimateForXInitialEtaWithTimeDelayWithControls}, we can deduce that there exists a constant $C\geq 0$, which does not depend on the controls and is allowed to vary from line to line, such that
	\begin{align*}
	  |I^1_\delta|^2
	  &\leq C\EXlr{\left.\big|W(t+\delta,X^{t,x;u^\vp,\beta[u^\vp]}_{t+\delta})-W(t+\delta,x)\big|^2\right|\calF_t}\\
	  &\leq C\EXlr{\left.\big|X^{t,x;u^\vp,\beta[u^\vp]}_{t+\delta}-x\big|^2+\big|X^{t,x;u^\vp,\beta[u^\vp]}_{t+\delta}-x\big|\right|\calF_t}
	  \leq C(\delta+\delta^{1/2}).
	\end{align*}
	From the definition of $G^{t,x;u^\vp,\beta[u^\vp]}_{t,t+\delta}[\cdot]$ we know that $I^2_\delta$ can be written as 
	\begin{align*}
	  I^2_\delta
	  ={}&\EX\bigg[W(t+\delta,x)+\int^{t+\delta}_tg(r,X^{t,x;u^\vp,\beta[u^\vp]}_r,\bar Y^{t,x;u^\vp,\beta[u^\vp]}_r,\bar Z^{t,x;u^\vp,\beta[u^\vp]}_r,u^\vp_r,\beta_r[u^\vp_\cdot])\dif r\\
	  &\quad+\int^{t+\delta}_tf(r,X^{t,x;u^\vp,\beta[u^\vp]}_r,\bar Y^{t,x;u^\vp,\beta[u^\vp]}_r,u^\vp_r,\beta_r[u^\vp_\cdot])\dif \eta^{t,x;u^\vp,\beta[u^\vp]}_r\bigg|\calF_t\bigg]-W(t+\delta,x).
	\end{align*}
	Noticing that $W(t+\delta,x)$ is deterministic, we derive by H\"older's inequality, linear growth for $g$ and $f$, \eqref{eq:EstimateForXEtaDominatedByZetaWithControls}, \eqref{eq:EstimateForXInitialEtaWithTimeDelayWithControls}, \cref{pro:EstimatesForSolutionsOfGBSDEsAppendix}  and boundedness of $\overline \calO$ that
	\begin{align*}
	  |I^2_\delta|^2
	  &\leq C\EX\bigg[\delta\int^{t+\delta}_t\big(1+|X^{t,x;u^\vp,\beta[u^\vp]}_r|^2+|\bar Y^{t,x;u^\vp,\beta[u^\vp]}_r|^2+|\bar Z^{t,x;u^\vp,\beta[u^\vp]}_r|^2\big)\bigg|\calF_t\bigg]\\
	  &\quad\;+C\EX\bigg[\eta^{t,x;u^\vp,\beta[u^\vp]}_{t+\delta}\cdot\int^{t+\delta}_t\big(1+|X^{t,x;u^\vp,\beta[u^\vp]}_r|^2+|\bar Y^{t,x;u^\vp,\beta[u^\vp]}_r|^2\big)\dif \eta^{t,x;u^\vp,\beta[u^\vp]}_r\bigg|\calF_t\bigg]\\
	  &\leq C\delta+C\EXlr{\left.\big|\eta^{t,x;u^\vp,\beta[u^\vp]}_{t+\delta}\big|^2\right|\calF_t}\leq C\delta.
	\end{align*}
	Thereby, we get that $W(t,x)-W(t+\delta,x)\leq C(\delta^{1/2}+\delta^{1/4})+\vp$. Then letting $\vp\to 0$ yields the desired results. The proof is completed.
\end{proof}
Observe that the continuity of $W(t,x)$ in $t$ implies the continuity of $J(t,x;u,v)$ in $t$. Then also by some approximation arguments we can get that $\prs$,  $J(\tau,\zeta;u,v)=Y^{\tau,\zeta;u,v}_\tau$ holds for each $\tau\in\calT_{t,T}$ and $\zeta\in L^4(\Omega,\calF_\tau,\PR;\overline \calO)$. Hence, the flow property of the backward semigroup $G$ can be written as
\begin{align}
  J(t,x;u,v)=Y^{t,x;u,v}_t&=G^{t,x;u,v}_{t,T}[\Phi(X^{t,x;u,v}_T)]=G^{t,x;u,v}_{t,\tau}[Y^{t,x;u,v}_\tau]\nonumber\\
  &=G^{t,x;u,v}_{t,\tau}[Y^{\tau,X^{t,x;u,v}_\tau;u,v}_\tau]=G^{t,x;u,v}_{t,\tau}[J(\tau,X^{t,x;u,v}_\tau;u,v)].\label{eq:FlowPropertyOfBackwardSemigropG}
\end{align} 

From \cref{thm:LowerValueFunctionWHolderContinuousInX} and \cref{thm:LowerValueFunctionWContinuousInT}, similarly to Proposition 2.6 in \citet*{WuYu2014SPA} and Theorem A.2 in \citet*{BuckdahnLi2008SICON}, we deduce the following conclusion. 
\begin{lem}\label{lem:LowerValueFunctionWithStoppingTimeAndRandomVariable}
	For each initial data $(t,x)\in\tT\times\overline \calO$, $\tau\in\calT_{t,T}$ and $\xi\in L^2(\Omega,\calF_\tau,\PR;\overline \calO)$, 
	\begin{equation*}
	  W(\tau,\xi)=\essinf_{\beta\in\calB_{\tau,T}}\esssup_{u\in\calU_{\tau,T}}J(\tau,\xi;u,\beta[u]),\quad \prs.
	\end{equation*}
\end{lem}

We are now ready to show the following (strong) DPP. The main difference between the weak and strong versions of DPP lies in that the intermediate time in the strong version is a random time $\tau\in\calT_{t,T}$, instead of the deterministic time $t+\delta$.
\begin{thm}[Strong DPP]\label{thm:DPPStrongVersion}
	Assume that \ref{H:BSigmaBoundedWithControls} -- \ref{H:GFMonotonicInYLipschitzInZWithControls} hold. For each $(t,x)\in\tT\times\overline \calO$ and $\tau\in\calT_{t,T}$, the lower value function $W(t,x)$ enjoys the following DPP:
	\begin{equation*}
	  W(t,x)=\essinf_{\beta\in\calB_{t,\tau}}\esssup_{u\in\calU_{t,\tau}}G^{t,x;u,\beta[u]}_{t,\tau}[W(\tau,X^{t,x;u,\beta[u]}_{\tau})],\quad \prs.
	\end{equation*}
\end{thm}
\begin{proof}
	This proof is similar to that of Theorem 3.6 in \citet*{BuckdahnLi2008SICON}, but it needs some necessary modifications. We denote the right hand side of desired identity by $W_\tau(t,x)$. 
	
	We first introduce a concatenation operation of controls. For each $\tau\in\calT_{t,T}$, $u_1(\cdot)\in\calU_{t,\tau}$ and $u_2(\cdot)\in\calU_{\tau,T}$, we define 
	\begin{equation*}
	  (u_1\oplus u_2)_s(\omega):=(u_1)_s(\omega)\one{[t,\tau(\omega)]}+(u_2)_s(\omega)\one{(\tau(\omega),T]},\quad s\in[t,T].
	\end{equation*}
	For any $\beta[\cdot]\in\calB_{t,T}$ and $u_2(\cdot)\in\calU_{\tau,T}$, we define a restriction $\beta_1[\cdot]$ of $\beta[\cdot]$ to $\calB_{t,\tau}$ as follows:
	\begin{equation*}
	  \beta_1[u_1]:=\beta[u_1\oplus u_2]\big|_{[t,\tau]},\quad u_1(\cdot)\in\calU_{t,\tau}.
	\end{equation*}
	Then we have $\beta_1[\cdot]\in\calB_{t,\tau}$. Similarly, we define the restriction $\beta_2[\cdot]\in\calB_{\tau,T}$ of $\beta[\cdot]\in\calB_{t,T}$, i.e., for each $u_1(\cdot)\in\calU_{t,\tau}$, 
	\begin{equation*}
	  \beta_2[u_2]:=\beta[u_1\oplus u_2]\big|_{(\tau,T]}, \quad u_2(\cdot)\in\calU_{\tau,T}.
	\end{equation*}
	The non-anticipativity of $\beta[\cdot]$ indicates that $\beta_1[\cdot]$ and $\beta_2[\cdot]$ are independent of the choice $u_2(\cdot)\in\calU_{\tau,T}$ and $u_1(\cdot)\in\calU_{t,\tau}$, respectively; moreover, we actually have $\beta[u_1\oplus u_2]=\beta_1[u_1]\oplus\beta_2[u_2]$ for any $u_1(\cdot)\in\calU_{t,\tau}$ and $u_2(\cdot)\in\calU_{\tau,T}$. 
	
	For each $(t,x)\in\tT\times\overline \calO$, $\tau\in\calT_{t,T}$, $u_1(\cdot)\in\calU_{t,\tau}$, $\beta_1[\cdot]\in\calB_{t,\tau}$ and $\beta_2[\cdot]\in\calB_{\tau,T}$, we define 
	\begin{equation*}
	  I_\tau(t,x,u_1,\beta_1,\beta_2):=\esssup_{u_2\in\calU_{\tau,T}}J(\tau,X^{t,x;u_1,\beta_1[u_1]}_\tau;u_2,\beta_2[u_2]).
	\end{equation*}
	Then by the definition of $W$ in \eqref{eq:DefOfValueFunctionWAndU} we get that \[W(\tau,X^{t,x;u_1,\beta_1[u_1]}_\tau)=\essinf_{\beta_2\in\calB_{\tau,T}}I_\tau(t,x,u_1,\beta_1,\beta_2),\quad\prs.\] 
	Next we start our proof and it will be divided into two steps.
	
	{\noindent\bfseries First Step:} $W(t,x)\leq W_\tau(t,x)$. 
	
	For each $(t,x)\in\tT\times\overline \calO$, $\tau\in\calT_{t,T}$, $u_1(\cdot)\in\calU_{t,\tau}$ and $\beta_1[\cdot]\in\calB_{t,\tau}$, there exists a sequence $\{\beta^j_2[\cdot]\}_{j\geq 1}\subset\calB_{\tau,T}$ such that 
	\begin{equation*}
	W(\tau,X^{t,x;u_1,\beta_1[u_1]}_\tau)=\inf_{j\geq 1}I_\tau(t,x,u_1,\beta_1,\beta^j_2),\quad \prs.
	\end{equation*}
	For each $\vp>0$ and $j\geq 1$, we put \begin{equation*}
	  \widetilde \Lambda_j:=\left\{I_\tau(t,x,u_1,\beta_1,\beta^j_2)\leq W(\tau,X^{t,x;u_1,\beta_1[u_1]}_\tau)+\vp\right\}.
	\end{equation*}
	Then $\widetilde \Lambda_j\in\calF_\tau$ for each $j\geq 1$, and $\Lambda_1:=\widetilde \Lambda_1$, $\Lambda_j:=\widetilde \Lambda_j/(\cup^{j-1}_{l=1}\widetilde \Lambda_l)$, $j\geq 2$ forms a partition of $(\Omega,\calF_\tau)$. We can also get that $\beta^\vp_2[\cdot]:=\sum_{j\geq 1}\one{\Lambda_j}\beta^j_2[\cdot]$ belongs to $\calB_{\tau,T}$. And the uniqueness of the solutions of RSDE \eqref{eq:RSDEControled} and GBSDE \eqref{eq:GBSDEControlled} implies that
	\begin{align*}
	  \sum_{j\geq 1}\one{\Lambda_j}I_\tau(t,x,u_1,\beta_1,\beta^j_2)
	  &=\esssup_{u_2\in\calU_{\tau,T}}\sum_{j\geq 1}\one{\Lambda_j}J(\tau,X^{t,x;u_1,\beta_1[u_1]}_\tau;u_2,\beta^j_2[u_2])\\
	  &=\esssup_{u_2\in\calU_{\tau,T}}J(\tau,X^{t,x;u_1,\beta_1[u_1]}_\tau;u_2,\beta^\vp_2[u_2])=I_\tau(t,x,u_1,\beta_1,\beta^\vp_2).
	\end{align*} 
	Thus, we conclude that for each $\vp>0$, $\tau\in\calT_{t,T}$, $u_1(\cdot)\in\calU_{t,\tau}$ and $\beta_1[\cdot]\in\calB_{t,\tau}$,
	\begin{equation}\label{eq:InequalityWTauXTauGeqIMinnusVp}
	W(\tau,X^{t,x;u_1,\beta_1[u_1]}_\tau)\geq \sum_{j\geq 1}\one{\Lambda_j}I_\tau(t,x,u_1,\beta_1,\beta^j_2)-\vp=I_\tau(t,x,u_1,\beta_1,\beta^\vp_2)-\vp.
	\end{equation}
	
	Next, for each $\vp>0$, $u(\cdot)\in\calU_{t,T}$ and $\beta_1[\cdot]\in\calB_{t,\tau}$, we define  $\beta[u]:=\beta_1[u_1]\oplus\beta^\vp_2[u_2]$, where $u_1:=u|_{[t,\tau]}$ and $u_2:=u|_{(\tau,T]}$. It is obvious that $\beta[\cdot]\in\calB_{t,T}$. It follows from \eqref{eq:DefOfValueFunctionWAndU} that for such defined $\beta[\cdot]\in\calB_{t,T}$, we have $\prs$, $W(t,x)\leq \esssup_{u\in\calU_{t,T}}J(t,x;u,\beta[u])$. 
	Then there exists a sequence $\{u^i(\cdot)\}_{i\geq 1}\subset\calU_{t,T}$ such that
	\begin{equation*}
    \esssup_{u\in\calU_{t,T}}J(t,x;u,\beta[u])=\sup_{i\geq 1}J(t,x;u^i,\beta[u^i]),\quad \prs.
	\end{equation*}
	For any $\vp>0$ and each $i\geq 1$, we put 
	\[\widetilde \Gamma_i:=\left\{\sup_{i\geq 1}J(t,x;u^i,\beta[u^i])\leq J(t,x;u^i,\beta[u^i])+\vp\right\}.\]
	Then $\widetilde \Gamma_i\in\calF_t$ for each $i\geq 1$, and $\Gamma_1:=\widetilde \Gamma_1$, $\Gamma_i:=\widetilde \Gamma_i/(\cup^{i-1}_{l=1}\widetilde \Gamma_l)\in\calF_t$, $i\geq 2$, forms a partition of $(\Omega,\calF_t)$. We also have that $u^\vp(\cdot):=\sum_{i\geq 1}\one{\Gamma_i}u^i(\cdot)$ belongs to $\calU_{t,T}$. Moreover, the non-anticipativity of $\beta[\cdot]$ and the uniqueness for solutions of RSDE \eqref{eq:RSDEControled} and GBSDE \eqref{eq:GBSDEControlled} yield that $\beta[u^\vp]=\sum_{i\geq 1}\one{\Gamma_i}\beta[u^i]$ and  $J(t,x;u^\vp,\beta[u^\vp])=\sum_{i\geq 1} \one{\Gamma_i} J(t,x;u^i,\beta[u^i])$, $\prs$. Hence, from \eqref{eq:FlowPropertyOfBackwardSemigropG} we deduce that 
	\begin{align*}
	  W(t,x)
	  &\leq \sum_{i\geq 1}\one{\Gamma_i}\sup_{i\geq 1}J(t,x;u^i,\beta[u^i])
  	\leq \sum_{i\geq 1}\one{\Gamma_i}J(t,x;u^i,\beta[u^i])+\vp\\
	  &=J(t,x;u^\vp,\beta[u^\vp])+\vp=G^{t,x;u^\vp,\beta[u^\vp]}_{t,\tau}[J(\tau,X^{t,x;u^\vp,\beta[u^\vp]}_\tau;u^\vp,\beta[u^\vp])]+\vp, \quad \prs,
	\end{align*} 
	where $\tau\in\calT_{t,T}$.  Note that $u^\vp=u^\vp_1\oplus u^\vp_2$ and $\beta[u^\vp]=\beta_1[u^\vp_1]\oplus\beta^\vp_2[u^\vp_2]$, where $u^\vp_1:=u^\vp|_{[t,\tau]}$, $u^\vp_2:=u^\vp|_{(\tau,T]}$ and $\beta_1[\cdot]\in\calB_{t,\tau}$. Then we have the following identity, 
	\begin{equation*}
	  G^{t,x;u^\vp,\beta[u^\vp]}_{t,\tau}[J(\tau,X^{t,x;u^\vp,\beta[u^\vp]}_\tau;u^\vp,\beta[u^\vp])]
	  =G^{t,x;u^\vp_1,\beta_1[u^\vp_1]}_{t,\tau}[J(\tau,X^{t,x;u^\vp_1,\beta_1[u^\vp_1]}_\tau;u^\vp_2,\beta_2[u^\vp_2])].
	\end{equation*}
	Thus, \cref{lem:ComparisonTheoremForSolutionsOfGBSDEs} and the inequality \eqref{eq:InequalityWTauXTauGeqIMinnusVp} yield that 
	\begin{align*}
	  W(t,x)
	  &\leq G^{t,x;u^\vp_1,\beta_1[u^\vp_1]}_{t,\tau}[I_\tau(t,x,u^\vp_1,\beta_1,\beta^\vp_2)]+\vp\\
	  &\leq G^{t,x;u^\vp_1,\beta_1[u^\vp_1]}_{t,\tau}[W(\tau,X^{t,x;u^\vp_1,\beta_1[u_1]}_\tau)+\vp]+\vp\\
	  &\leq G^{t,x;u^\vp_1,\beta_1[u^\vp_1]}_{t,\tau}[W(\tau,X^{t,x;u^\vp_1,\beta_1[u_1]}_\tau)]+(C+1)\vp,\quad\prs.
	\end{align*}
	Therefore, by the definition of $W_\tau(t,x)$ we obtain that $\prs$, $W(t,x)\leq W_\tau(t,x)+(C+1)\vp$. Finally, sending $\vp\to0$ yields the desired result.
	
	{\noindent\bfseries Second Step: }$W(t,x)\geq W_\tau(t,x)$.
	
	For each $(t,x)\in\tT\times\overline \calO$, $\tau\in\calT_{t,T}$, $u_1(\cdot)\in\calU_{t,\tau}$, $\beta_1[\cdot]\in\calB_{t,\tau}$ and $\beta_2[\cdot]\in\calB_{\tau,T}$, there exists a sequence $\{u^i_2(\cdot)\}_{i\geq 1}\subset\calU_{\tau,T}$ such that 
	$I_\tau(t,x,u_1,\beta_1,\beta_2)=\sup_{i\geq 1}J(\tau,X^{t,x;u_1,\beta_1[u_1]}_\tau;u^i_2,\beta_2[u^i_2])$. For each $\vp>0$ and $i\geq 1$, we put 
	\begin{equation*}
	  \widetilde \varGamma_i:=\left\{I_\tau(t,x,u_1,\beta_1,\beta_2)\leq J(\tau,X^{t,x;u_1,\beta_1[u_1]}_\tau;u^i_2,\beta_2[u^i_2])+\vp\right\}.
	\end{equation*} 
	Then $\widetilde \varGamma_i\in\calF_\tau$ for each $i\geq 1$, and $\varGamma_1:=\widetilde \varGamma_1$, $\varGamma_i:=\widetilde \varGamma_i/(\cup^{i-1}_{l=1}\widetilde \varGamma_i)$, $i\geq 2$ forms a partition of $(\Omega,\calF_\tau)$. We can also get that $u^\vp_2(\cdot):=\sum_{i\geq 1}\one{\varGamma_i}u^i_2(\cdot)$ belongs to $\calU_{\tau,T}$, $\beta_2[u^\vp_2]=\sum_{i\geq 1}\one{\varGamma_i}\beta_2[u^i_2]$. 
	Thus, the uniqueness for solution of GBSDE implies that
	\begin{align*}
	  I_\tau(t,x,u_1,\beta_1,\beta_2)
	  \leq \sum_{i\geq 1}\one{\varGamma_i}J(\tau,X^{t,x;u_1,\beta_1[u_1]}_\tau;u^i_2,\beta_2[u^i_2])+\vp
	  =J(\tau,X^{t,x;u_1,\beta_1[u_1]}_\tau;u^\vp_2,\beta_2[u^\vp_2])+\vp.
	\end{align*}
	
	It follows from the definition of $W(t,x)$ in \eqref{eq:DefOfValueFunctionWAndU} that there exists a sequence $\{\beta^i[\cdot]\}_{i\geq 1}\subset\calB_{t,T}$ satisfying 
	\begin{equation*}
	  W(t,x)=\inf_{i\geq 1}\esssup_{u\in\calU_{t,T}} J(t,x;u,\beta^i[u]).
	\end{equation*} 
	For each $\vp>0$ and $i\geq 1$, we put 
	\begin{equation*}
	  \widetilde \varLambda_i:=\bigg\{\esssup_{u(\cdot)\in\calU_{t,T}}J(t,x;u,\beta^i[u])-\vp\leq W(t,x)\bigg\}.
	\end{equation*}
	Then $\widetilde{\varLambda}_i\in\calF_t$ for each $i\geq 1$, and $\varLambda_1:=\widetilde \varLambda_1$, $\varLambda_i:=\widetilde \varLambda_i/(\cup^{i-1}_{l=1}\widetilde \varLambda_l)$, $i\geq 2$ forms a partition of $(\Omega,\calF_t)$. We also have that $\beta^\vp[\cdot]:=\sum_{i\geq 1}\one{\varLambda_i}\beta_i[\cdot]$ belongs to $\calB_{t,T}$. And the uniqueness of RSDE \eqref{eq:RSDEControled} and GBSDE \eqref{eq:GBSDEControlled} implies that $J(t,x;u,\beta^\vp[u])=\sum_{i\geq 1}\one{\varLambda_i}J(t,x;u,\beta^i[u])$. Thus, \eqref{eq:FlowPropertyOfBackwardSemigropG} indicates that, for each $\tau\in\calT_{t,T}$,
	\begin{align*}
	  W(t,x)
	  &\geq \sum_{i\geq 1}\one{\varLambda_i}\esssup_{u\in\calU_{t,T}}J(t,x;u,\beta^i[u])-\vp
	  \geq \sum_{i\geq1}\one{\varLambda_i}J(t,x;u,\beta^i[u])-\vp\\
	  &=J(t,x;u,\beta^\vp[u])-\vp
	  =G^{t,x;u,\beta^\vp[u]}_{t,\tau}[J(\tau,X^{t,x;u,\beta^\vp[u]}_\tau;u,\beta^\vp[u])]-\vp,\quad \prs.
	\end{align*}
	Take $u:=u_1\oplus u^\vp_2$ for each $u_1(\cdot)\in\calU_{t,\tau}$. Let $\beta^\vp_1[\cdot]$ and $\beta^\vp_2[\cdot]$ be the restriction of $\beta^\vp[\cdot]$ to $\calB_{t,\tau}$ and $\calB_{\tau,T}$, respectively. Then we have $\beta^\vp[u]=\beta^\vp[u_1\oplus u^\vp_2]=\beta^\vp_1[u_1]\oplus\beta^\vp_2[u^\vp_2]$. Hence, we obtain the following identity,
	\begin{equation*}
	  G^{t,x;u,\beta^\vp[u]}_{t,\tau}[J(\tau,X^{t,x;u,\beta^\vp[u]}_\tau;u,\beta^\vp[u])]
	  =G^{t,x;u_1,\beta^\vp_1[u_1]}_{t,\tau}[J(\tau,X^{t,x;u_1,\beta^\vp_1[u_1]}_\tau;u^\vp_2,\beta^\vp_2[u^\vp_2])].
	\end{equation*}
	Then, we deduce that 
	\begin{align*}
	  W(t,x)
	  &\geq G^{t,x;u_1,\beta^\vp_1[u_1]}_{t,\tau}[I_\tau(t,x,\beta^\vp_1,\beta^\vp_2)-\vp]-\vp\\
	  &\geq G^{t,x;u_1,\beta^\vp_1[u_1]}_{t,\tau}[\essinf_{\beta_2\in\calB_{t,\tau}}I_\tau(t,x,\beta^\vp_1,\beta_2)]-(C+1)\vp\\
	  &=G^{t,x;u_1,\beta^\vp_1[u_1]}_{t,\tau}[W(\tau,X^{t,x;u_1,\beta^\vp_1[u_1]}_\tau)]-(C+1)\vp, \quad \prs.
	\end{align*}
	Therefore, by the definition of $W_\tau(t,x)$ we derive that $\prs$, $W(t,x)\geq W_\tau(t,x)-(C+1)\vp$. Then sending $\vp\to0$ yields that desired result.
\end{proof}

\section{Viscosity solutions of Isaacs equations}
\label{sec:ViscositySolutionOfHJBI}
This section aims at proving the lower value function $W(t,x)$ and upper value function $U(t,x)$ are, respectively, the unique viscosity solution of following Hamilton-Jacobi-Bellman-Isaacs equations:
\begin{equation}\label{eq:HJBIEquationWithLowerValueFunctionW}
  \begin{cases}
    \displaystyle 
    \partial_tW(t,x)+H^-(t,x,W,\nabla W,D^2W)=0,&(t,x)\in[0,T)\times\calO,\\
    \displaystyle 
    \frac{\partial}{\partial n}W(t,x)+\sup_{u\in U}\inf_{v\in V}f(t,x,W(t,x),u,v)=0,& (t,x)\in[0,T)\times\partial\calO,\\
    \displaystyle 
    W(T,x)=\Phi(x),& x\in\overline{\calO},
  \end{cases}
\end{equation}
and 
\begin{equation}\label{eq:HJBIEquationWithUpperValueFunctionW}
\begin{cases}
\displaystyle 
\partial_tU(t,x)+H^+(t,x,U,\nabla U,D^2U)=0,&(t,x)\in[0,T)\times\calO,\\
\displaystyle 
\frac{\partial}{\partial n}U(t,x)+\inf_{v\in V}\sup_{u\in U}f(t,x,U(t,x),u,v)=0,& (t,x)\in[0,T)\times\partial\calO,\\
\displaystyle 
U(T,x)=\Phi(x),& x\in\overline{\calO},
\end{cases}
\end{equation}
where the Hamiltonians and operator $\partial/\partial n$ are defined as follows, for each $t\in\tT$, $x\in\overline \calO$, $y\in\rtn$, $p\in\rtn^n$ and $A\in S^n$ with $S^n$ being the set of all $n\times n$ symmetric matrices,
\begin{align*}
  H^-(t,x,y,p,A)
  &:=\sup_{u\in U}\inf_{v\in V}\Big\{\frac{1}{2}Tr\{\sigma\sigma^*(t,x,u,v)A\}+\langle b(t,x,u,v),p\rangle+g(t,x,y,\sigma^*p,u,v)\Big\},\\
  H^+(t,x,y,p,A)
  &:=\inf_{v\in V}\sup_{u\in U}\Big\{\frac{1}{2}Tr\{\sigma\sigma^*(t,x,u,v)A\}+\langle b(t,x,u,v),p\rangle+g(t,x,y,\sigma^*p,u,v)\Big\},\\
  \frac{\partial}{\partial n}
  &:=\sum^{n}_{i=1}\frac{\partial\phi}{\partial x_i}(x)\frac{\partial}{\partial x_i}.
\end{align*}
For this purpose, we provide a new approach --- the representation theorem for generators of GBSDEs, instead of Peng's approximation method introduced by \citet*{Peng1997SuiJiFenXiXuanJiang}. The representation theorem approach is more convenient. 

Let us recall the viscosity solution of PDE \eqref{eq:HJBIEquationWithLowerValueFunctionW}, which is adapted from \citet*{CrandallIshiiLions1992BAMS} and \citet*{BuckdahnLi2008SICON}. The counterpart of PDE \eqref{eq:HJBIEquationWithUpperValueFunctionW} is analogous.
\begin{dfn}\label{dfn:ViscositySolutionOfHJBIEquationsWithLowerValueFunction}
  A function $W\in C(\tT\times\overline \calO;\rtn)$ is called a viscosity sub- (resp., super-) solution of \eqref{eq:HJBIEquationWithLowerValueFunctionW} if $W(T,x)\leq \Phi(x)$ for all $x\in\overline \calO$ (resp., $W(T,x)\geq \Phi(x)$), and for any $\varphi\in C^{1,2}(\tT\times\overline \calO;\rtn)$ such that whenever $(t,x)\in[0,T)\times\overline \calO$ is global maximum (resp., minimum) of $W-\varphi$, we have 
  \begin{equation}\label{eq:InequalityViscositySubsolutionOfHJBI}
    \begin{cases}
    \displaystyle
    \partial_t\varphi(t,x)+H^-(t,x,W,\nabla\varphi,D^2\varphi)\geq 0,\quad x\in\calO;\\
    \displaystyle
    [\partial_t\varphi(t,x)+H^-(t,x,W,\nabla\varphi,D^2\varphi)]\vee\bigg[\frac{\partial\varphi}{\partial n}(t,x)+\sup_{u\in U}\inf_{v\in V}f(t,x,W,u,v)\bigg]\geq 0,\  x\in\partial\calO,
    \end{cases}
  \end{equation}
  \begin{equation}\label{eq:InequalityViscositySupersolutionOfHJBI}
  \left(\text{resp.,}
    \begin{cases}
    \displaystyle 
    \partial_t\varphi(t,x)+H^-(t,x,W,\nabla\varphi,D^2\varphi)\leq 0,\quad x\in\calO;\\
    \displaystyle 
    [\partial_t\varphi(t,x)\!+H^-(t,x,W,\nabla\varphi,D^2\varphi)]\wedge\bigg[\frac{\partial\varphi}{\partial n}(t,x)\!+\sup_{u\in U}\inf_{v\in V}f(t,x,W,u,v)\bigg]\!\!\leq 0,\  x\in\partial\calO\!\!
    \end{cases}\!\!\!\!\!\right)\!.
  \end{equation}
  Moreover, a function $W\in C(\tT\times\overline \calO;\rtn)$ is called a viscosity solution of \eqref{eq:HJBIEquationWithLowerValueFunctionW} if it is both a viscosity subsolution and a viscosity supersolution. 
\end{dfn}

\subsection{Viscosity solution of Isaacs equation: Existence result}
Here we only prove that the lower value function $W(t,x)$ is a viscosity solution of PDE \eqref{eq:HJBIEquationWithLowerValueFunctionW} since the proof of $U(t,x)$ being a viscosity solution of PDE \eqref{eq:HJBIEquationWithUpperValueFunctionW} is symmetric. The representation theorem for generators of GBSDEs play an essential role in proof. 
\begin{thm}\label{thm:ExistenceResultForViscositySOlutionsOfHJBIEquations}
	Assume that \ref{H:BSigmaBoundedWithControls} -- \ref{H:GFMonotonicInYLipschitzInZWithControls} hold. Then the lower value function $W(t,x)$ is a viscosity solution of PDE \eqref{eq:HJBIEquationWithLowerValueFunctionW}. 
\end{thm}
\begin{proof}
	The continuity of $W(t,x)$ in $(t,x)$ follows from \cref{thm:LowerValueFunctionWHolderContinuousInX,thm:LowerValueFunctionWContinuousInT}. For given initial data $(t,x)\in[0,T)\times\overline \calO$, we define $\psi:\Omega\times[t,T]\mapsto\rtn^+$ as $\psi_s:=\eta^{t,x;u,v}_s+s$. Its inverse function is denoted by $\tau_\cdot$. We know that for each given $r\in[t,\psi_T]$, $\tau_r\in\calT_{t,T}$. 
	
	{\bfseries First Step.} This step aims to prove that $W(t,x)$ is a viscosity subsolution of \eqref{eq:HJBIEquationWithLowerValueFunctionW}. Take any $\varphi\in C^{1,2}(\tT\times\overline \calO;\rtn)$ and $(t,x)\in[0,T)\times\overline \calO$ such that $W-\varphi$ achieves the global maximum at $(t,x)$. Without loss of generality, we assume $W(t,x)=\varphi(t,x)$. Since $W(T,x)=\Phi(x)$ is trivially satisfied for all $x\in\overline \calO$, we only need to prove \eqref{eq:InequalityViscositySubsolutionOfHJBI}. 
	
	It follows from \cref{thm:DPPStrongVersion} that for each $0<\delta\leq T-t$, 
	\begin{align*}
	  \varphi(t,x)=W(t,x)
	  =\essinf_{\beta\in\calB_{t,\tau_{t+\delta}}}\esssup_{u\in\calU_{t,\tau_{t+\delta}}}G^{t,x;u,\beta[u]}_{t,\tau_{t+\delta}}[W(\tau,X^{t,x;u,\beta[u]}_{\tau_{t+\delta}})],\quad \prs.
	\end{align*}
	The fact $W\leq\varphi$ and \cref{lem:ComparisonTheoremForSolutionsOfGBSDEs} imply that 
	\begin{equation}\label{eq:InfSupGVarphiXMinusVarphi}
	  \essinf_{\beta\in\calB_{t,\tau_{t+\delta}}}\esssup_{u\in\calU_{t,\tau_{t+\delta}}}\left\{G^{t,x;u,\beta[u]}_{t,\tau_{t+\delta}}[\varphi(\tau_{t+\delta},X^{t,x;u,\beta[u]}_{\tau_{t+\delta}})]-\varphi(t,x)\right\}\geq 0,\quad\prs.
	\end{equation}
	For each $u(\cdot)\in\calU_{t,\tau_{t+\delta}}$ and $v(\cdot)\in\calV_{t,\tau_{t+\delta}}$, we denote $Y^{u,v,\delta}_t:=G^{t,x;u,v}_{t,\tau_{t+\delta}}[\varphi(\tau_{t+\delta},X^{t,x;u,v}_{\tau_{t+\delta}})]$, which is a solution of the following GBSDE,
	\begin{align*}
	  Y^{u,v,\delta}_t={}&\varphi(\tau_{t+\delta},X^{t,x;u,v}_{\tau_{t+\delta}})+\int^{\tau_{t+\delta}}_tg(r,X^{t,x;u,v}_r,Y^{u,v,\delta}_r,Z^{u,v,\delta}_r,u_r,v_r)\dif r\\
	  &+\int^{\tau_{t+\delta}}_tf(r,X^{t,x;u,v}_r,Y^{u,v,\delta}_r,u_r,v_r)\dif\eta^{t,x;u,v}_r-\int^{\tau_{t+\delta}}_t\langle Z^{t,x;u,v}_r,\dif B_r\rangle.
	\end{align*}
	It\^o's formula to $\varphi(r,X^{t,x;u,v}_r)$ yields that
	\begin{align*}
	  \varphi(t,x)={}&\varphi(\tau_{t+\delta},X^{t,x;u,v}_{\tau_{t+\delta}})
	  -\int^{\tau_{t+\delta}}_t\langle \nabla\varphi(r,,X^{t,x;u,v}_r),\sigma(r,X^{t,x;u,v}_r,u_r,v_r)\dif B_r\rangle
	  -\int^{\tau_{t+\delta}}_t\partial_r\varphi(r,X^{t,x;u,v}_r)\dif r\\
	  &\!\!-\int^{\tau_{t+\delta}}_t\!\!\Big[\frac{1}{2}Tr\{\sigma\sigma^*(r,X^{t,x;u,v}_r\!,u_r,v_r)D^2\varphi(r,X^{t,x;u,v}_r)\}\!+\!\langle b(r,X^{t,x;u,v}_r\!,u_r,v_r),\nabla\varphi(r,X^{t,x;u,v}_r)\rangle\Big]\!\dif r\\
	  &\!\!-\int^{\tau_{t+\delta}}_t\frac{\partial\varphi}{\partial n}(r,X^{t,x;u,v}_r)\one{\partial \calO}(X^{t,x;u,v}_r)\dif\eta^{t,x;u,v}_r.
	\end{align*}
	Next we set
	\[\hat Y^{u,v,\delta}_\cdot:=Y^{u,v,\delta}_\cdot-\varphi(\cdot,X^{t,x;u,v}_\cdot), \quad
	\hat Z^{u,v,\delta}_\cdot:=Z^{u,v,\delta}_\cdot-\sigma^*(\cdot,X^{t,x;u,v}_\cdot,u_\cdot,v_\cdot)\nabla\varphi(\cdot,X^{t,x;u,v}_\cdot).\]
	Hence, we deduce that 
	\begin{align}
	  \hat Y^{u,v,\delta}_t={}&\int^{\tau_{t+\delta}}_tG(r,X^{t,x;u,v}_r,\hat Y^{u,v,\delta}_r,\hat Z^{u,v,\delta}_r,u_r,v_r)\dif r\nonumber\\
	  &+\int^{\tau_{t+\delta}}_tF(r,X^{t,x;u,v}_r,\hat Y^{u,v,\delta}_r,u_r,v_r)\one{\partial\calO}(X^{t,x;u,v}_r)\dif\eta^{t,x;u,v}_r-\int^{\tau_{t+\delta}}_r\langle \hat Z^{u,v,\delta}_r,\dif B_r\rangle,\label{eq:GBSDEForIsaacsEqViscositySolutionEquvToBackwardGroup}
	\end{align}
	where for each $r\in\tT$, $\theta\in M$, $u\in U$ and $v\in V$,
	\begin{align*}
	  G(r,\theta,u,v):={}&\partial_r\varphi(r,x)+g(r,x,y+\varphi(r,x),z+\sigma^*(r,x,u,v)\nabla\varphi(r,x),u,v)\\
	  &+\frac{1}{2}Tr\{\sigma\sigma^*(r,x,u,v)D^2\varphi(r,x)\}+\langle b(r,x,u,v),\nabla\varphi(r,x)\rangle,\\
	  F(r,x,y,u,v):={}&f(r,x,y+\varphi(r,x),u,v)+\frac{\partial\varphi}{\partial n}(t,x).
	\end{align*}
	Then \eqref{eq:RepThmForGBSDEIntegralType} in \cref{thm:RepresentationTheoremForGenerator} indicates that there exists a pair of real-valued positive processes $(a_\cdot,b_\cdot)$ with $a_\cdot+b_\cdot=1$ and $a_\cdot>0$ such that for $\dte$ $t\in[0,T)$, 
	\begin{equation}\label{eq:RepresentationAlphaGAndBetaFEquivHatY}
	\lim_{\delta\to0}\frac{1}{\delta}\EX\bigg[\hat Y^{u,v,\delta}_t -\int^{t+\delta}_th(\tau_r,X^{t,x;u,v}_{\tau_r},u_{\tau_r},v_{\tau_r})\dif r\bigg]=0,
	\end{equation}
	where we set
	\[h(\tau_r,X^{t,x;u,v}_{\tau_r}\!,u_{\tau_r},v_{\tau_r}):=a_{\tau_r}G(\tau_r,X^{t,x;u,v}_{\tau_r}\!,0,0,u_{\tau_r},v_{\tau_r})+b_{\tau_r}F(\tau_r,X^{t,x;u,v}_{\tau_r}\!,0,u_{\tau_r},v_{\tau_r}).\]
	Thus, for each $\vp_1>0$ there exists a small enough $\delta>0$ such that 
	\begin{equation}\label{eq:InequalityAfterRepThmForIsaacsEqVsicositySubsolution}
	\frac{1}{\delta}\EX\left[\hat Y^{u,v,\delta}_t\right]\leq\EX\bigg[\frac{1}{\delta}\int^{t+\delta}_th(\tau_r,X^{t,x;u,v}_{\tau_r},u_{\tau_r},v_{\tau_r})\dif r\bigg]+\vp_1.
	\end{equation}
	On the other hand, according to  \eqref{eq:InfSupGVarphiXMinusVarphi} we know that for each $0<\delta\leq T-t$, 
	\begin{equation*}
	\essinf_{\beta\in\calB_{t,\tau_{t+\delta}}}\esssup_{u\in\calU_{t,\tau_{t+\delta}}}\hat Y^{u,\beta[u],\delta}_t\geq 0, \quad\prs.
	\end{equation*}
	Analogous to the arguments of proofs in \cref{thm:DPPStrongVersion}, we can obtain that for each $\beta[\cdot]\in\calB_{t,\tau_{t+\delta}}$and $\vp>0$, there exists a  $u^\vp(\cdot)\in\calU_{t,\tau_{t+\delta}}$ such that 
	\begin{equation*}
	\hat Y^{u^\vp,\beta[u^\vp],\delta}_t\geq-\delta\vp,\quad \prs.
	\end{equation*}
	Plugging the previous inequality into \eqref{eq:InequalityAfterRepThmForIsaacsEqVsicositySubsolution} yields that for each $\vp_1>0$, there exists a small enough $0<\delta\leq T-t$ such that the following inequality
	\begin{equation}\label{eq:InequalityForContradictionIsaacsViscositySubsolution}
	-\vp\leq\EX\bigg[\frac{1}{\delta}\int^{t+\delta}_th(\tau_r,X^{t,x;u^\vp,\beta[u^\vp]}_{\tau_r},u^\vp_{\tau_r},\beta_{\tau_r}[u^\vp_{\tau_\cdot}])\dif r\bigg]+\vp_1
	\end{equation}
	holds for each $\beta[\cdot]\in\calB_{t,\tau_{t+\delta}}$ and $\vp>0$. 
	
	We first consider the case $(t,x)\in[0,T)\times\partial\calO$. In this case we need to prove 
	\begin{equation*}
	  [\partial_t\varphi(t,x)+H^-(t,x,W,D\varphi,D^2\varphi)]\vee\left[\frac{\partial\varphi}{\partial n}(t,x)+\sup_{u\in U}\inf_{v\in V}f(t,x,W,u,v)\right]\geq 0.
	\end{equation*}  
	Let us suppose that this is not true. Then there exists a $\vp_0>0$ such that
	\begin{equation*}
	  \sup_{u\in U}\inf_{v\in V}G(t,x,0,0,u,v)<-\vp_0<0,\quad 
	  \sup_{u\in U}\inf_{v\in V}F(t,x,0,u,v)<-\vp_0<0.
	\end{equation*}
	Thus, we can find a measurable function $\bar\beta:U\mapsto V$ such that, for all $u\in U$, 
	\begin{equation*}
	  G(t,x,0,0,u,\bar\beta[u])<-\vp_0<0,\quad \frac{\partial\varphi}{\partial n}(t,x)+f(t,x,W(t,x),u,\bar\beta[u])<-\vp_0<0.
	\end{equation*} 
	By putting $\bar\beta_s[u](\omega):=\bar\beta[u_s(\omega)]$, $(s,\omega)\in[t,\tau_{t+\delta}]\times\Omega$, we identify $\bar\beta$ as an element of $\calB_{t,\tau_{t+\delta}}$. Then taking $v_t(\omega):=\bar\beta_t[u](\omega)$ in  \eqref{eq:InequalityForContradictionIsaacsViscositySubsolution}, and considering that $h(r,x,u,v)$ is continuous and of linear growth in $x$, we can deduce by \cref{pro:ApproximationForOverlineG} that there exists a $(\calF_{\tau_t})$-progressively measurable process sequence $\{(h^n_t)_{t\in\tT}\}^\infty_{n=1}$ such that
	\begin{equation}\label{eq:InequalityForContradictionIsaacsViscositySubsolutionBeforeLimit}
	-\vp\leq\EX\bigg[\frac{1}{\delta}\int^{t+\delta}_t2n|X^{t,x;u^\vp,\bar \beta[u^\vp]}_{\tau_r}-x|+h^n_r+h(\tau_r,x,u^\vp_{\tau_r},\bar\beta_{\tau_r}[u^\vp_{\tau_\cdot}])\dif r\bigg]+\vp_1.
	\end{equation}
	It follows from \cref{pro:EstimateForRSDEHatXEta} that 
	\begin{align}
	&\frac{2n}{\delta}\EX\bigg[\int^{t+\delta}_t|X^{t,x;u^\vp,\bar \beta[u^\vp]}_{\tau_r}-x|\dif r\bigg]
	\leq 2n\EX\bigg[\sup_{r\in[t,t+\delta]}|X^{t,x;u^\vp,\bar \beta[u^\vp]}_{\tau_r}-x|\bigg]\nonumber\\
	&\quad=2n\EX\bigg[\sup_{r\in[t,\tau_{t+\delta}]}|X^{t,x;u^\vp,\bar \beta[u^\vp]}_r-x|\bigg]
	\leq2n\EX\bigg[\sup_{r\in[t,t+\delta]}|X^{t,x;u^\vp,\bar \beta[u^\vp]}_r-x|\bigg]
	\leq 2n\delta^{1\over2}.\label{eq:EstimateForXTaurForIsaacsViscosity}
	\end{align}
	And we can derive from Proposition 2.2 in \citet*{Jiang2008AAP} that for each $n\geq 0$ and $\dte$ $t\in\tT$, 
	\begin{equation}\label{eq:HnSequenceLebesgueLemmaForIsaacsEq}
	\lim_{\delta\to0}\EX\bigg[\frac{1}{\delta}\int^{t+\delta}_th^n_r\dif r\bigg]=\EX[h^n_t].
	\end{equation}
	Then \cref{pro:ApproximationForOverlineG} indicates that $\EX[h^n_t]\to 0$ when $n\to\infty$. Moreover, since $G(\cdot,x,0,0,\cdot,\cdot)$ and  $F(\cdot,x,0,\cdot,\cdot)$ are uniformly continuous in $\tT\times U\times V$, there exists a $\delta_1>0$ with $\delta_1\leq T-t$ such that for each $u\in U$, $t\leq s\leq t+\delta_1$,
	\begin{gather*}
	G(s,x,0,0,u,\bar\beta[u])\leq -\frac{\vp_0}{2},\quad
	F(s,x,0,u,\bar\beta[u])\leq -\frac{\vp_0}{2},
	\end{gather*}
	whence we have
	\[h(s,x,u,\bar{\beta}[u])\leq-\frac{\vp_0}{2}.\]
	After taking $0<\delta\leq\delta_1$, we get that $t\leq \tau_r\leq r\leq t+\delta\leq t+\delta_1$, 
	\[h(\tau_r,x,u^\vp_{\tau_r},\bar\beta_{\tau_r}[u^\vp_\cdot])<-\frac{\vp_0}{2}, \quad\prs.\]
	Hence, we take $\vp_1=\vp_0/4$ in  \eqref{eq:InequalityForContradictionIsaacsViscositySubsolutionBeforeLimit}, and then send $\delta\to0$, $n\to\infty$ and $\vp\to0$, obtaining $\vp_0\leq 0$, which contradicts with $\vp_0>0$. Therefore, the desired conclusion holds true. 
	
	Now we consider the case $(t,x)\in[0,T)\times\calO$, in which $F$, $\eta^{t,x;u,v}_\cdot$ and $b_\cdot$ vanish and $a_\cdot\equiv 1$. It only needs to prove 
	$\partial_t\varphi(t,x)+H^-(t,x,W(t,x),\nabla\varphi,D^2\varphi)\geq 0$. If this is not true, we can still obtain the previous contradiction. So the desired result follows. 
	
	{\bfseries Second Step.} Now we prove that $W(t,x)$ is a viscosity supersolution of PDE \eqref{eq:HJBIEquationWithLowerValueFunctionW}. Take any $\varphi\in C^{1,2}(\tT\times\overline \calO;\rtn)$ and $(t,x)\in[0,T)\times\overline \calO$ such that $W-\varphi$ achieves the global minimum $0$ at $(t,x)$. We only need to prove \eqref{eq:InequalityViscositySupersolutionOfHJBI}. Similar to the first step, for each $0<\delta\leq T-t$, we have that 
	\begin{equation}\label{eq:InfSupGVarphiXMinusVarphiLeqZero}
	\essinf_{\beta\in\calB_{t,\tau_{t+\delta}}}\esssup_{u\in\calU_{t,\tau_{t+\delta}}}\left\{G^{t,x;u,\beta[u]}_{t,\tau_{t+\delta}}[\varphi(\tau_{t+\delta},X^{t,x;u,\beta[u]}_{\tau_{t+\delta}})]-\varphi(t,x)\right\}\leq 0, \quad \prs,
	\end{equation}
	and the equation \eqref{eq:GBSDEForIsaacsEqViscositySolutionEquvToBackwardGroup} and identity \eqref{eq:RepresentationAlphaGAndBetaFEquivHatY} still hold. Then for each $\vp_1>0$ there exists a small enough $\delta>0$ such that 
	\begin{equation}\label{eq:InequalityAfterRepThmForIsaacsEqVsicositySupersolution}
	\frac{1}{\delta}\EX\left[\hat Y^{u,v,\delta}_t\right]
	\geq\EX\bigg[\frac{1}{\delta}\int^{t+\delta}_th(\tau_r,X^{t,x;u,v}_{\tau_r},u_{\tau_r},v_{\tau_r})\dif r\bigg]-\vp_1.
	\end{equation}
	Furthermore, it follows from the inequality  \eqref{eq:InfSupGVarphiXMinusVarphiLeqZero} that for each $0<\delta\leq T-t$, 
	\begin{equation*}
	\essinf_{\beta\in\calB_{t,\tau_{t+\delta}}}\esssup_{u\in\calU_{t,\tau_{t+\delta}}}\hat Y^{u,\beta[u],\delta}_t\leq 0,\quad \prs.
	\end{equation*}
	Analogous to the arguments of proofs in \cref{thm:DPPStrongVersion}, we can obtain that for each $\vp>0$, there exists a  $\beta^\vp[\cdot]\in\calB_{t,\tau_{t+\delta}}$ such that $\prs$, $\hat Y^{u,\beta^\vp[u],\delta}_t\leq\delta\vp$ holds for each $u(\cdot)\in\calU_{t,\tau_{t+\delta}}$. Plugging the previous inequality into  \eqref{eq:InequalityAfterRepThmForIsaacsEqVsicositySupersolution} yields that for each $\vp_1>0$ there exists a small enough $0<\delta\leq T-t$ such that 
	\begin{equation}\label{eq:InequalityForContradictionIsaacsViscositySupersolution}
	\vp\geq\EX\bigg[\frac{1}{\delta}\int^{t+\delta}_th(\tau_r,X^{t,x;u,\beta^\vp[u]}_{\tau_r},u_{\tau_r},\beta^\vp_{\tau_r}[u_\cdot])\dif r\bigg]-\vp_1
	\end{equation}
	holds for each $u(\cdot)\in\calU_{t,\tau_{t+\delta}}$ and $\vp>0$. 
	
	We first consider the case of $(t,x)\in[0,T)\times\partial\calO$. In this case we only need to prove  
	\begin{equation*}
	[\partial_t\varphi(t,x)+H^-(t,x,W,\nabla\varphi,D^2\varphi)]\wedge\left[\frac{\partial\varphi}{\partial n}(t,x)+\sup_{u\in U}\inf_{v\in V}f(t,x,W,u,v)\right]\leq 0.
	\end{equation*}  
	We assume that the previous inequality does not hold. Then there exists a $\vp_0>0$ such that,
	\begin{equation*}
	\sup_{u\in U}\inf_{v\in V}G(t,x,0,0,u,v)>\vp_0,\quad 
	\sup_{u\in U}\inf_{v\in V}F(t,x,0,u,v)>\vp_0.
	\end{equation*}
	Thus, there exists a $\bar u\in U$ such that for each $\beta[\cdot]\in\calB_{t,\tau_{t+\delta}}$, 
	\begin{equation*}
	G(t,x,0,0,\bar u,\beta[\bar u])>\vp_0,\quad 
	F(t,x,0,\bar u,\beta[\bar u])>\vp_0.
	\end{equation*} 
	Since $G(\cdot,x,0,0,\cdot,\cdot)$ and $F(\cdot,x,0,\cdot,\cdot)$ is uniformly continuous in $\tT\times U\times V$, there exists a $\delta_1>0$ with $\delta_1\leq T-t$ such that for each  $\beta[\cdot]\in\calB_{t,\tau_{t+\delta}}$ and $t\leq s\leq t+\delta_1$,
	\begin{gather*}
	G(s,x,0,0,\bar u,\beta[\bar u])\geq \frac{\vp_0}{2},\quad
	F(s,x,0,\bar u,\beta[\bar u])\geq \frac{\vp_0}{2},
	\end{gather*}
	thereby, we get that
	\[h(s,x,\bar u,\beta[\bar u])\geq\frac{\vp_0}{2}.\]
	Taking $u(\cdot):=\bar u$ in \eqref{eq:InequalityForContradictionIsaacsViscositySupersolution}, and considering that $h(r,x,u,v)$ is continuous and of linear growth in $x$, we can deduce from \cref{pro:ApproximationForOverlineG} that there exists a $(\calF_{\tau_t})$-progressively measurable process sequence  $\{(h^n_t)_{t\in\tT}\}^\infty_{n=1}$ such that for small enough $0<\delta\leq \delta_1$,
	\begin{equation*}
	\vp\geq\EX\bigg[\frac{1}{\delta}\int^{t+\delta}_t-2n|X^{t,x;\bar u,\beta^\vp[\bar u]}_{\tau_r}-x|-h^n_r+h(\tau_r,x,\bar u,\beta^\vp_{\tau_r}[\bar u])\dif r\bigg]-\vp_1.
	\end{equation*}
	Applying \eqref{eq:EstimateForXTaurForIsaacsViscosity} and \eqref{eq:HnSequenceLebesgueLemmaForIsaacsEq}, and sending $\delta\to0$ imply that for  $\dte$ $t\in\tT$,
	\begin{equation*}
	\vp\geq -\EX[h^n_t]+\frac{\vp_0}{2}-\vp_1.
	\end{equation*} 
	We take $\vp_1=\vp_0/4$, and send $n\to\infty$ and $\vp\to0$, obtaining  $\vp_0\leq 0$, which contradicts with $\vp_0>0$. Then the desired result holds. 
	
	Now we consider the case of $(t,x)\in[0,T)\times\calO$, in which $F$, $\eta^{t,x;u,v}_\cdot$ and $b_\cdot$ vanish, and $a_\cdot\equiv 1$. It reduces to prove $\partial_t\varphi(t,x)+H^-(t,x,W(t,x),\nabla\varphi,D^2\varphi)\leq 0$.	If this is not true, we can still obtain the previous contradiction. Therefore, $W(t,x)$ is a viscosity supersolution of PDE \eqref{eq:HJBIEquationWithLowerValueFunctionW}. 
\end{proof}

\subsection{Viscosity solution of Isaacs equation: Uniqueness result}
This subsection provides the uniqueness for the viscosity solution of PDEs \eqref{eq:HJBIEquationWithLowerValueFunctionW} and \eqref{eq:HJBIEquationWithUpperValueFunctionW}. To obtain the uniqueness we only need to prove a comparison theorem for viscosity subsolutions and supersolutions. For this purpose, we will adapt some methods of \citet*{BarlesBuckdahnPardoux1997SSR}, \citet*{MaCvitanic2001JAMSA} and \citet*{BuckdahnLi2008SICON} to our settings. Here we only consider PDE \eqref{eq:HJBIEquationWithLowerValueFunctionW} since the case of PDE \eqref{eq:HJBIEquationWithUpperValueFunctionW} is analogous. 
 
\begin{thm}\label{thm:ComparisonTheoremForViscositySubSuperSolutions}
	Assume that \ref{H:BSigmaBoundedWithControls} -- \ref{H:GFMonotonicInYLipschitzInZWithControls} hold, and $g(t,\theta,u,v)$ is nondecreasing in $y$ for each $(t,x,z,u,v)\in\tT\times\overline \calO\times\rtn^d\times U\times V$. Let $w_1$ and $w_2$ be, respectively, a viscosity subsolution and supersolution of PDE \eqref{eq:HJBIEquationWithLowerValueFunctionW}. Then we have $w_1(t,x)\leq w_2(t,x)$ for all $(t,x)\in\tT\times\overline \calO$.
\end{thm}

\begin{proof}
	Since both of $w_1$ and $w_2$ are continuous and $\overline\calO$ is compact, we only need to show $w_1(t,x)\leq w_2(t,x)$ in $\tT\times\calO$. Define a subset of $\calO$ as $\calO^\alpha:=\{x\in\overline \calO:d(x,\partial\calO)\geq\alpha\}$ for each $\alpha>0$. We choose a $\alpha_0>0$ such that $\calO^\alpha\neq\varnothing$ for each $0<\alpha\leq \alpha_0$. It then reduces to prove that for each $0<\alpha\leq\alpha_0$, $w_1(t,x)\leq w_2(t,x)$ for all $(t,x)\in\tT\times\calO^\alpha$. We first show the following auxiliary lemma.
	\begin{lem}\label{lem:AuxiliaryPDEInUniquenessViscosityProof}
		Suppose that the assumptions of \cref{thm:ComparisonTheoremForViscositySubSuperSolutions} are in force. Then for each $0<\alpha\leq\alpha_0$, the function $w(t,x)=w_1(t,x)-w_2(t,x)$, $(t,x)\in\tT\times\calO^\alpha$ is a viscosity subsolution of the following PDE:
		\begin{equation*}
		  \begin{cases}
		    \partial_t w(t,x)+H_w(t,x,w,\nabla w,D^2w)=0,\quad (t,x)\in[0,T)\times\calO^\alpha,\\
		    w(T,x)=0,\quad x\in\calO^\alpha,
		  \end{cases}
		\end{equation*}
		where for each $(t,x,w,p,A)\in\tT\times\calO^\alpha\times\rtn\times\rtn^n\times S^n$,
		\[H_w(t,x,w,p,A):=\sup_{u\in U,v\in V}\bigg\{\frac{1}{2}Tr\{\sigma\sigma^*(t,x,u,v)A\}+\langle b(t,x,u,v),p\rangle+\widetilde K|w|+\widetilde K|p||\sigma(t,x,u,v)|\bigg\},\]
		and $\widetilde K$ is a positive constant depending on $K$ and $\lambda_1$.
	\end{lem}
  \begin{proof}[{\bf Proof of \cref{lem:AuxiliaryPDEInUniquenessViscosityProof}}]
  	Let us fix arbitrarily a $\alpha\in(0,\alpha_0]$. Take any $\varphi\in C^{1,2}(\tT\times\calO^\alpha;\rtn)$ such that $(t_0,x_0)\in[0,T)\times\calO^\alpha$ is a global maximum point of $w-\varphi$. Note that $w_1(T,x)\leq \Phi(x)$ and $w_2(T,x)\geq\Phi(x)$. Then we have that $w(T,x)\leq 0$. By the definition of viscosity subsolution, it is sufficient to prove the following inequality,
  	\begin{equation*}
  	  \partial_tw(t_0,x_0)+H_w(t_0,x_0,w,\nabla w,D^2w)|_{(t,x)=(t_0,x_0)}\geq 0.
  	\end{equation*}
  	Next we introduce the following notations of parabolic second order semijets of function $w$ at $(t,x)$. 
  	\begin{align*}
  	  \calP^{2,+}w(t,x):={}&\{(\partial_t\varphi(t,x),\nabla\varphi(t,x),D^2\varphi(t,x)):\varphi\in C^{1,2}(\tT\times\calO^\alpha;\rtn)\\
  	  & \text{ and } w-\varphi \text{ has a global maximum at } (t,x)\};\\
  	  \calP^{2,-}w(t,x):={}&\{(\partial_t\varphi(t,x),\nabla\varphi(t,x),D^2\varphi(t,x)):\varphi\in C^{1,2}(\tT\times\calO^\alpha;\rtn)\\ &\text{ and } w-\varphi \text{ has a global minimum at } (t,x)\}.
  	\end{align*}
  	With the same notations, we record the closures of the superjets $\calP^{2,+}$ as follows, and the closures of the subjets $\calP^{2,-}$ can be defined similarly.
  	\begin{align*}
  	\overline \calP^{2,+}w(t,x)={}&\{(p,q,X)\in\rtn\times\rtn^n\times S^n: \text{ there exists a sequence  } (t_n,x_n,p_n,q_n,X_n)\in\\ &[0,T)\times \calO^\alpha\times\rtn\times\rtn^n\times S^n \text{ such that } (p_n,q_n,X_n)\in\calP^{2,+}w(t_n,x_n) \\
  	&\text{ and } (t_n,x_n,p_n,q_n,X_n)\to(t,x,p,q,X)\}.
  	\end{align*}
  	Now let us define a function as follows, for each $\vp>0$, $t\in\tT$ and $x$, $y\in\calO^\alpha$,
  	\begin{equation*}
  	  \psi_\vp(t,x,y):=w_1(t,x)-w_2(t,y)-\frac{|x-y|^2}{\vp}-\varphi(t,x).
  	\end{equation*}
  	We suppose that $(t^\vp,x^\vp,y^\vp)$ is the maximum point of $\psi_\vp$ for each $\vp>0$. Then Proposition 3.7 of \citet*{CrandallIshiiLions1992BAMS} implies that $(t^\vp,x^\vp,y^\vp)$ tends to $(t_0,x_0,x_0)$ and $|x^\vp-y^\vp|^2/\vp$ tends to $0$ as $\vp\to0$. Moreover, Theorem 8.3 in \citet*{CrandallIshiiLions1992BAMS} indicates that there exist $(X,Y)\in S^n\times S^n$ and $a\in\rtn^n$ such that
  	\begin{equation*}
  	  (a+\partial_t\varphi(t^\vp,x^\vp),p^\vp+\nabla\varphi(t^\vp,x^\vp),X)\in\overline \calP^{2,+}w_1(t^\vp,x^\vp),\quad (a,p^\vp,Y)\in\overline \calP^{2,-}w_2(t^\vp,y^\vp),
  	\end{equation*}
  	\begin{equation}\label{eq:MartixInequalityForHamiltonPDE}
  	  \begin{bmatrix}
      	X & 0\\
      	0 &-Y
    	\end{bmatrix}
  	  \leq \frac{2}{\vp}
  	  \begin{bmatrix}
      	I & -I\\
      	-I & I
    	\end{bmatrix}+
    	\begin{bmatrix}
      	D^2\varphi(t^\vp,x^\vp) & 0\\
       	0 & 0
    	\end{bmatrix},
  	\end{equation}
  	where $p^\vp:=2(x^\vp-y^\vp)/\vp$ and $I$ denotes the identity matrix. Then the definitions of semijets yield that
  	\begin{equation}\label{eq:HamiltonPDESubSuperSolution}
    	\begin{cases}
  	  a+\partial_t\varphi(t^\vp,x^\vp)+F(t^\vp,x^\vp,w_1(t^\vp,x^\vp),p^\vp+\nabla\varphi(t^\vp,x^\vp),X)\geq 0,\\
  	  a+F(t^\vp,y^\vp,w_2(t^\vp,y^\vp),p^\vp,Y)\leq 0.
    	\end{cases}
  	\end{equation}
  	For brevity, we denote, for each $u\in U$ and $v\in V$,
  	\begin{align*}
  	  I^{\vp,u,v}_1&:=\frac{1}{2}Tr\{\sigma\sigma^*(t^\vp,x^\vp,u,v)X\}-\frac{1}{2}Tr\{\sigma\sigma^*(t^\vp,y^\vp,u,v)Y\},\\
  	  I^{\vp,u,v}_2&:=\langle b(t^\vp\,x^\vp,u,v),p^\vp+\nabla\varphi(t^\vp,x^\vp)\rangle-\langle b(t^\vp,y^\vp,u,v),p^\vp\rangle,\\
  	  I^{\vp,u,v}_3&:=g(t^\vp,x^\vp,w_1(t^\vp,x^\vp),\sigma^*(t^\vp,x^\vp,u,v)(p^\vp+\nabla\varphi(t^\vp,x^\vp)),u,v)-g(t^\vp,y^\vp,w_2(t^\vp,y^\vp),\sigma^*(t^\vp,y^\vp)p^\vp,u,v).
  	\end{align*}
  	Subtracting the second inequality from the first one in \eqref{eq:HamiltonPDESubSuperSolution} yields that 
  	\begin{equation}\label{eq:InequalitySubMinusSuperSolutionsWithControls}
  	  \partial_t\varphi(t^\vp,x^\vp)+\sup_{u\in U,v\in V}\left\{I^{\vp,u,v}_1+I^{\vp,u,v}_2+I^{\vp,u,v}_2\right\}\geq 0.
  	\end{equation}
  	Then it follows from the inequality \eqref{eq:MartixInequalityForHamiltonPDE} and \ref{H:BSigmaLipschitzInXWithControls} that 
  	\begin{equation*}
  	  I^{\vp,u,v}_1\leq \frac{1}{2}Tr\{\sigma\sigma^*(t^\vp,x^\vp,u,v)D^2\varphi(t^\vp,x^\vp)\}+\frac{|x^\vp-y^\vp|^2}{\vp}.
  	\end{equation*}
  	And from \ref{H:BSigmaLipschitzInXWithControls} we can derive that
  	\begin{align*}
  	  I^\vp_2={}&\langle b(t^\vp,x^\vp,u,v)-b(t^\vp,y^\vp,u,v),p^\vp\rangle+\langle b(t^\vp,x^\vp,u,v),\nabla\varphi(t^\vp,x^\vp)\rangle\nonumber\\
  	  \leq{}& K\frac{|x^\vp-y^\vp|^2}{\vp}+\langle b(t^\vp,x^\vp,u,v),\nabla\varphi(t^\vp,x^\vp)\rangle.
  	\end{align*}
  	Since $g$ satisfies the monotonicity condition in \ref{H:ItemGFMonotonicInYWithControls} and $g(t,x,y,z,u,v)$ is nondecreasing in $y$, we have 
  	\[g(t,x,y_1,z,u,v)-g(t,x,y_2,z,u,v)\leq |\lambda_1||y_1-y_2|.\] 
  	Hence, the previous inequality, \ref{H:ItemPhiGFLipschitzInZWithControls} and \ref{H:BSigmaLipschitzInXWithControls} imply that
  	\begin{align*}
  	  I^{\vp,u,v}_3&\leq K|x^\vp-y^\vp|+|\lambda_1||w_1(t^\vp,x^\vp)-w_2(t^\vp,y^\vp)|+K^2\frac{|x^\vp-y^\vp|^2}{\vp}+K|\sigma(t^\vp,x^\vp,u,v)||\nabla\varphi(t^\vp,x^\vp)|.
  	\end{align*}
  	Thereby, plugging the estimates of $I^{\vp,u,v}_i$ $(i=1,2,3)$ into \eqref{eq:InequalitySubMinusSuperSolutionsWithControls} and then sending $\vp\to0$, we conclude that 
  	\begin{align*}
  	  &\partial_t\varphi(t_0,x_0)+\sup_{u\in U,v\in V}\Big\{\frac{1}{2}Tr\{\sigma\sigma^*(t_0,x_0,u,v)D^2\varphi(t_0,x_0)\}+\langle b(t_0,x_0,u,v),\nabla\varphi(t_0,x_0)\rangle\\
    	&+|\lambda_1||w(t_0,x_0)|+K|\sigma(t_0,x_0,u,v)||\nabla\varphi(t_0,x_0)|\Big\}\geq 0.
  	\end{align*}
  	After choosing $\widetilde K:=K+|\lambda_1|$, we obtain the desired result.
  \end{proof}
  
  {\noindent\bf Coming back to the proof of \cref{thm:ComparisonTheoremForViscositySubSuperSolutions}.}  With \cref{lem:AuxiliaryPDEInUniquenessViscosityProof} in hand, we only need to prove that for any $0< \alpha\leq\alpha_0$, $w(t,x)\leq 0$ holds for each $(t,x)\in\tT\times\calO^\alpha$. Let us fix arbitrarily a $\alpha\in(0,\alpha_0]$. For each $(t,x)\in\tT\times\calO^\alpha$, we define the function $\chi(t,x):=\exp\{(C_1(T-t)+1)\psi(x)\}$, where $\psi(x):=(\log\sqrt{1+|x|^2}+1)^2$ and constant $C_1>0$ will be chosen later. By some direct computations we have the following estimates of the first and second derivatives of $\psi(x)$,
  \begin{equation*}
  |\nabla\psi(x)|\leq\frac{2\sqrt{\psi(x)}}{\sqrt{1+|x|^2}}\leq 4; \qquad
  |D^2\psi(x)|\leq\frac{2\sqrt{\psi(x)}}{1+|x|^2}\leq 4.
  \end{equation*}
  In what follows, we denote $t_1:=T-1/C_1$. Then the previous estimates yield that, for each $t\in[t_1,T]$,
  \begin{align*}
  \partial_t\chi(t,x)&=-C_1\psi(x)\chi(t,x);\quad |\nabla\chi(t,x)|=(C_1(T-t)+1)\chi(t,x)|\nabla\psi(t,x)|\leq 8\chi(t,x),\\
  |D^2\chi(t,x)|&=|(C_1(T-t)+1)^2\chi(t,x)\nabla\psi(x)(\nabla\psi(x))^*+(C_1(T-t)+1)\chi(t,x)D^2\psi(x)|
  \leq 72\chi(t,x).
  \end{align*}
  Noticing that $b$ and $\sigma$ satisfy \ref{H:BSigmaBoundedWithControls} and $\psi(x)$, $\chi(t,x)\geq 1$. We denote the boundedness of $b$ and $\sigma$ by $C\geq 0$. Then we can obtain that for each $(t,x)\in[t_1,T]\times\calO^\alpha$,
  \begin{align}
  & \partial_t\chi(t,x)+ H_w(t,x,\chi,\partial_t\chi,\nabla\chi,D^2\chi) \nonumber\\
  &\leq -C_1\psi(x)\chi(t,x) +\frac{1}{2}C^2|D^2\chi(t,x)|+C|\nabla\chi(t,x)|+\widetilde K\chi(t,x)+\widetilde K C|\nabla\chi(t,x)|\nonumber\\
  &\leq \chi(t,x)\{-C_1+36C^2+8C(1+\widetilde K)+\widetilde K\}<0,\label{eq:HamiltonPDEPluggedInChi}
  \end{align}
  provided $C_1$ is large enough. 
  
  For each $\vp>0$, we define 
  \begin{equation*}
  M^\alpha(\vp):=\max_{[t_1,T]\times\calO^\alpha}\big\{\big(w(t,x)-\vp\chi(t,x)\big)\me^{-\widetilde K(T-t)}\big\}.
  \end{equation*}
  In the sequel, we will prove that $M^\alpha(\vp)\leq0$. Note that  $[t_1,T]\times\calO^\alpha$ is compact, $w(t,x)$ and $\chi(t,x)$ are continuous. Then for each $\vp>0$, there exists $(t^\vp,x^\vp)\in[t_1,T]\times\calO^\alpha$ attains the maximum point $M^\alpha(\vp)$. For each $(t,x)\in[t_1,T]\times\calO^\alpha$, define $\varphi(t,x):=\vp\chi(t,x)+M^\alpha(\vp)\exp\{\widetilde K(T-t)\}$. Then $\varphi\in C^{1,2}([t_1,T]\times\calO^\alpha;\rtn)$, $\varphi(t^\vp,x^\vp)=w(t^\vp,x^\vp)$ and $w(t,x)-\varphi(t,x)\leq 0$ for all $(t,x)\in[t_1,T]\times\calO^\alpha$. 
  We suppose that $M^\alpha(\vp)>0$ for some $\vp>0$. Then $w(t^\vp,x^\vp)=\varphi(t^\vp,x^\vp)>0$, and by the definition of viscosity subsolution we have \[\partial_t\varphi(t^\vp,x^\vp)+H_w(t^\vp,x^\vp,w,\nabla\varphi,D^2\varphi)|_{(t,x)=(t^\vp,x^\vp)}\geq 0.\] 
  On the other hand, by some direct calculations and \eqref{eq:HamiltonPDEPluggedInChi} we derive that 
  \begin{align*}
    &\partial_t\varphi(t^\vp,x^\vp)+H_w(t^\vp,x^\vp,w,\nabla\varphi,D^2\varphi)|_{(t,x)=(t^\vp,x^\vp)}\\
    &=\vp\big(\partial_t\chi(t^\vp,x^\vp)+ H_w(t^\vp,x^\vp,\chi,\partial_t\chi,\nabla\chi,D^2\chi)|_{(t,x)=(t^\vp,x^\vp)}\big)<0.
  \end{align*}
  This contradiction indicates that $M^\alpha(\vp)\leq 0$ holds for each $\vp>0$. Thus, we have that $w(t,x)\leq \vp\chi(t,x)$ for all $t\in[t_1,T]\times\calO^\alpha$. Sending $\vp\to0$ yields that $w(t,x)\leq 0$ for all $t\in[t_1,T]\times\calO^\alpha$. Finally, applying the same arguments in the intervals $[t_{i+1},t_i]$ $(i=1,2,\cdots)$, where $t_{i+1}:=(t_i-1/C_1)^+$. Therefore, we obtain that for any $0<\alpha\leq\alpha_0$, $w(t,x)\leq 0$ holds for all $(t,x)\in\tT\times\calO^\alpha$. 
\end{proof}
  
\begin{rmk}
	Analogous arguments will yield that the upper value function $U(t,x)$ is the unique viscosity solution of Isaacs equation \eqref{eq:HJBIEquationWithUpperValueFunctionW}. Since $H^-\leq H^+$ holds, any viscosity solution of PDE \eqref{eq:HJBIEquationWithUpperValueFunctionW} is a viscosity supersolution of PDE \eqref{eq:HJBIEquationWithLowerValueFunctionW}. Then \cref{thm:ComparisonTheoremForViscositySubSuperSolutions} indicates that $W(t,x)\leq U(t,x)$. Moreover, if the Isaacs' condition holds, i.e., if for each $(t,x,y,p,A)\in\tT\times\rtn^n\times\rtn\times\rtn^n\times S^n$, it holds that
	\[H^-(t,x,y,p,A)=H^+(t,x,y,p,A),\]
	then PDEs \eqref{eq:HJBIEquationWithLowerValueFunctionW} and \eqref{eq:HJBIEquationWithUpperValueFunctionW} coincide, and the uniqueness for viscosity solution implies that the lower value function $W(t,x)$ equals the upper value function $U(t,x)$. This suggests that the associated stochastic differential game with state constraints admits a value.
\end{rmk}

\section{Appendix: Complementary results}
\label{sec:AppendixComplementaryResults}
This section provides some estimates for solutions of RSDEs and GBSDEs with general coefficients, and their detailed proofs. We first introduce the following RSDE: 
\begin{equation}\label{eq:RSDEsInAppendix}
\begin{cases}
  \displaystyle X^{t,\zeta}_s=\zeta+\int^s_t b(r,X^{t,\zeta}_r)\dif r+\int^s_t\sigma(r,X^{t,\zeta}_r)\dif B_r+\int^s_t\nabla\phi(X^{t,\zeta}_r)\dif\eta^{t,\zeta}_r,\quad s\in[t,T];\\
  \displaystyle\eta^{t,\zeta}_s=\int^s_t\one{\partial\calO}(X^{t,\zeta}_r)\dif\eta^{t,\zeta}_r,\quad \eta^{t,\zeta}_\cdot \text{ is increasing,}
\end{cases}
\end{equation}
where the initial time $t\in\tT$ and initial state $\zeta\in L^4(\Omega,\calF_T,\PR;\overline \calO)$ are given, $b:\Omega\times\tT\times\overline \calO\mapsto\rtn^n$, $\sigma:\Omega\times\tT\times\overline \calO\mapsto\rtn^{n\times d}$, and $b(\cdot,\cdot,x)$ and $\sigma(\cdot,\cdot,x)$ are both $(\calF_t)$-progressively measurable for each $x\in\overline \calO$. A strong solution to the above RSDE is a pair of adapted processes $(X^{t,\zeta}_s,\eta^{t,\zeta}_s)_{s\in[t,T]}$ valued in $\overline \calO\times\rtn^+$ and satisfies this RSDE almost surely for all $s\in[t,T]$. \citet*{Marin-RubioReal2004JTP} gives the existence and uniqueness for solutions of RSDE \eqref{eq:RSDEsInAppendix} under the following assumptions.
\begin{enumerate}
	\renewcommand{\theenumi}{(A\arabic{enumi})}
	\renewcommand{\labelenumi}{\theenumi}
	\setcounter{enumi}{2}
	\item\label{A:BSigmaBoundedAppendix} $b$ and $\sigma$ are uniformly bounded;
	\item\label{A:BSigmaLipschitzAppendix} There exists a constant $K\geq 0$ such that $\pts$, for each $x_1$ and $x_2\in\overline \calO$, 
	\[|b(t,x_1)-b(t,x_2)|+|\sigma(t,x_1)-\sigma(t,x_2)|\leq K|x_1-x_2|.\]
\end{enumerate}
\begin{pro}\label{pro:EstimateForRSDEHatXEta}
	Assume that \ref{A:BSigmaBoundedAppendix} -- \ref{A:BSigmaLipschitzAppendix} hold. For each $s\leq t$, we set $X^{t,\zeta}_s:=\zeta$ and $\eta^{t,\zeta}_s=0$. For each $t$, $t'\in\tT$, $\zeta$, $\zeta'\in L^4(\Omega,\calF_{t\wedge t'},\PR;\overline \calO)$, there exists a constant $C\geq 0$ depending on $K$, $T$, $\phi$, $b$ and $\sigma$ such that $\prs$,
	\begin{equation*}
	\EX\bigg[\sup_{s\in[0,T]}\big|X^{t,\zeta}_s-X^{t',\zeta'}_s\big|^4+\sup_{s\in[0,T]}\big|\eta^{t,\zeta}_s-\eta^{t',\zeta'}_s\big|^4\bigg|\calF_{t\wedge t'}\bigg]\leq C(|\zeta-\zeta'|^4+|t-t'|^2).
	\end{equation*}
	\begin{equation*}
	\EX\bigg[\sup_{s\in[t,T]}\big|X^{t,\zeta}_s\big|^4+\sup_{s\in[t,T]}\big|\eta^{t,\zeta}_s\big|^4\bigg|\calF_t\bigg]\leq C(1+|\zeta|^4).
	\end{equation*}
	Moreover, for each $\lambda>0$, $s\in[t,T]$, there exists a constant $C_{\lambda,s}$ such that for each $\zeta\in L^4(\Omega,\calF_t,\PR;\overline \calO)$,
	\begin{equation*}
	  \prs,\quad \EXlr{\left.\me^{\lambda \eta^{t,\zeta}_s}\right|\calF_t}\leq C_{\lambda,s}.
	\end{equation*}
	
\end{pro}
\begin{proof}
	For brevity, we set $\Psi_\cdot:=\Psi^{t,\zeta}_\cdot$, $\Psi'_\cdot:=\Psi^{t',\zeta'}_\cdot$ and $\hat \Psi_\cdot:=\Psi_\cdot-\Psi'_\cdot$, where $\Psi=X$, $\eta$. Denote $\hat\phi_\cdot:=\me^{-C_0[\phi(X_\cdot)+\phi(X'_\cdot)]}$. Without loss of generality, we just consider the case $t\leq t'$. 
	
	Firstly, if $0\leq t\leq t'\leq s\leq T$, we have that 
	\begin{align*}
	  \hat X_s={}&X_{t'}-\zeta'+\int^s_{t'}(b(r,X_r)-b(r,X'_r))\dif r +\int^s_{t'}(\sigma(r,X_r)-\sigma(r,X'_r))\dif B_r\\
	  &+\int^s_{t'}\nabla\phi(X_r)\dif \eta_r-\int^s_{t'}\nabla\phi(X'_r)\dif \eta'_r.
	\end{align*}
	It\^o's formula to $\hat\phi_r|\hat X_r|^2$ on the interval $[t',T]$ yields that
	\begin{align}
	\hat \phi_s|\hat X_s|^2
	={}&\hat \phi_{t'}|X_{t'}-\zeta'|^2+\int^s_{t'}\hat \phi_r\Big[2\langle\hat X_r,b(r,X_r)-b(t,X'_r)\rangle\dif r+|\sigma(r,X_r)-\sigma(r,X'_r)|^2\dif r\nonumber\\
	&+2\langle\hat X_r,(\sigma(r,X_r)-\sigma(r,X'_r))\dif B_r\rangle+2\langle\hat X_r,\nabla\phi(X_r)\dif\eta_r-\nabla\phi(X'_r)\dif\eta'_r\rangle\Big]\nonumber\\
	&-C_0\int^s_{t'}\hat\phi_r|\hat X_r|^2 \Big[\langle\nabla\phi(X_r),b(r,X_r)\rangle\dif r+\langle\nabla\phi(X'_r),b(r,X'_r)\rangle\dif r\nonumber\\
	&\;\;+\!\langle\nabla\phi(X_r),\sigma(r,X_r)\!\dif B_r\rangle\!+\!\langle\nabla\phi(X'_r),\sigma(r,X'_r)\!\dif B_r\rangle\!+\!|\nabla\phi(X_r)|^2\!\dif\eta_r\!+\!|\nabla\phi(X'_r)|^2\!\dif\eta'_r\nonumber\\
	&\;\;+\frac{1}{2}Tr\left\{\sigma\sigma^*(r,X_r)D^2\phi(X_r)+\sigma\sigma^*(r,X'_r)D^2\phi(X'_r)\right\}\dif r\Big]\nonumber\\
	&+\frac{{C_0}^2}{2}\int^s_{t'}\hat \phi_r|\hat X_r|^2|\sigma^*(r,X_r)\nabla\phi(X_r)+\sigma^*(r,X'_r)\nabla\phi(X'_r)|^2\dif r\nonumber\\
	&-2C_0\int^s_{t'}\hat\phi_r\hat X^*_r\big(\sigma(r,X_r)-\sigma(r,X'_r)\big)\big(\sigma^*(r,X_r)\nabla\phi(X_r\!)+\sigma^*(r,X'_r)\nabla\phi(X'_r)\big)\dif r.\label{eq:ItoFormulaToRSDEXMinusX'}
	\end{align}
	Note that $|\nabla\phi(x)|=1$ for each $x\in\partial\calO$. If follows from \eqref{eq:InequalityOfBoundedSetO1} that, for each $s\in[t',T]$,
	\begin{equation*}
	\int^s_{t'}\hat \phi_r\left[2\langle\hat X_r,\nabla\phi(X_r)\rangle-C_0|\hat X_r|^2|\nabla\phi(X_r)|^2\right]\dif\eta_r\leq 0,
	\end{equation*}
	\begin{equation*}
	\int^s_{t'}\hat\phi_r\left[-2\langle\hat X_r,\nabla\phi(X'_r)\rangle-C_0|\hat X_r|^2|\nabla\phi(X'_r)|^2\right]\dif\eta'_r\leq 0.
	\end{equation*}
	Plugging the previous two inequalities into \eqref{eq:ItoFormulaToRSDEXMinusX'}	and noticing that $b$, $\sigma$, $\phi$, $\nabla\phi$ and $D^2\phi$ are uniformly bounded, we can deduce that 
	\begin{align*}
	|\hat X_s|^2
	\leq{}&C\bigg\{|X_{t'}-\zeta'|^2+\int^s_t|\hat X_r|^2\dif r+\int^s_t\hat\phi_r\langle\hat X_r,(\sigma(r,X_r)-\sigma(r,X'_r))\dif B_r\rangle\\
	&\qquad-\int^s_t\hat\phi_r|\hat X_r|^2 \left[\langle\nabla\phi(X_r),\sigma(r,X_r)\dif B_r\rangle+\langle\nabla\phi(X'_r),\sigma(r,X'_r)\dif B_r\rangle\right]\bigg\},
	\end{align*}
	where $C\geq 0$ is a generic constant depending on $K$, $b$, $\sigma$ and $\phi$, and it allowed to vary from line to line. Squaring and taking supremum with respect to $s\in[t',T]$, we can deduce by the BDG, H\"older and Gronwall inequalities that $\prs$,
	\begin{equation}\label{eq:RSDEHatXEstimateFromtToT'}
	  \EX\bigg[\sup_{s\in[t',T]}|\hat X_s|^4\bigg|\calF_t\bigg]\leq C\EXlr{\left.|X_{t'}-\zeta'|^4\right|\calF_t}.
	\end{equation}
	
	Secondly, if $0\leq t\leq s\leq t'\leq T$, we have that, for each $s\in[t,t']$, 
	\begin{equation*}
	\hat X_s=\zeta-\zeta'+\int^s_tb(r,X_r)\dif r+\int^s_t\sigma(r,X_r)\dif B_r+\int^s_t\nabla\phi(X_r)\dif\eta_r.
	\end{equation*}
	Analogous to the previous arguments, applying It\^o's formula to $\hat \phi_r|\hat X_r|^2$ in the interval $[t,t']$, inequality \eqref{eq:InequalityOfBoundedSetO1}, the fact $|\sigma(t,x)|\leq K(1+|x|)$, and noticing that $b$, $\sigma$, $\phi$, $\nabla\phi$ and $D^2\phi$ are uniformly bounded, we can conclude that 
  \begin{equation*}
    |\hat X_s|^2
    \leq C\bigg\{|\zeta-\zeta'|^2+|t-t'|+\int^s_t|\hat X_r|^2\dif r+\int^s_t\hat\phi_r\langle\hat X_r+|\hat X_r|^2 \nabla\phi(X_r),\sigma(r,X_r)\dif B_r\rangle\bigg\},
  \end{equation*}
  where $C\geq 0$ also depends on $T$. Square and take supremum with respect to $s\in[t,t']$ in both sides, then the BDG, H\"older and Gronwall inequalities imply that $\prs$,
  \begin{equation}\label{eq:RSDEHatXEstimateFromtTot'}
    \EX\bigg[\sup_{s\in[t,t']}\big|\hat X_s\big|^4\bigg|\calF_t\bigg]\leq C(|\zeta-\zeta'|^4+|t-t'|^2).
  \end{equation}
  The previous inequality indicates that $\prs$, $\EX[|X_{t'}-\zeta'|^4|\calF_t]\leq C(|\zeta-\zeta'|^4+|t-t'|^2)$, which together with \eqref{eq:RSDEHatXEstimateFromtToT'} suggests that $\EX[\sup_{s\in[t',T]}|\hat X_{s}|^4|\calF_t]\leq C(|\zeta-\zeta'|^4+|t-t'|^2)$. 
  
  Thirdly, if $0\leq s\leq t\leq t'\leq T$, we have $X_s=\zeta$ and $X'_s=\zeta'$, which means that the desired result is trivial. Thereby, we obtain the following inequality, $\prs$,
  \begin{equation*}
    \EX\bigg[\sup_{s\in[0,T]}|\hat X_s|^4\bigg|\calF_t\bigg]\leq C(|\zeta-\zeta'|^4+|t-t'|^2).
  \end{equation*}
  
	Moreover, It\^o's formula to $\phi(X_r)$ on the interval $[t,s]$ yields that, for each $s\in[t,T]$, 
  \begin{align*}
    \eta_s\!=\!\phi(X_s)\!-\!\phi(\zeta)\!-\!\!\int^s_t\bigg[\frac{1}{2}Tr\big\{D^2\phi(X_r)\sigma\sigma^*(r,X_r)\big\}+\langle\nabla \phi(X_r),b(r,X_r)\rangle\bigg]\!\!\dif r \!-\!\!\int^s_t\langle\nabla\phi(X_r),\sigma(r,X_r)\!\dif B_r\rangle.
  \end{align*}
  Note that $\phi$, $\nabla\phi$, $D^2\phi$, $b$ and $\sigma$ are Lipschitz continuous and bounded. It is easily follows from the previous identity that for each $\lambda>0$ and $s\in[t,T]$, there exists a constant $C_{\lambda,s}$ such that $\prs$,
  \begin{equation*}
    \EX\bigg[\sup_{s\in[0,T]}\big|\eta^{t,\zeta}_s-\eta^{t',\zeta'}_s\big|^4\bigg|\calF_t\bigg]\leq C(|\zeta-\zeta'|^4+|t-t'|^2);\quad \EXlr{\left.\me^{\lambda \eta^{t,\zeta}_s}\right|\calF_t}\leq C_{\lambda,s}.
  \end{equation*}	
	Then the desired result follows.
\end{proof}

Next we construct some a priori estimates for solutions of the following GBSDE, 
\begin{equation}\label{eq:GBSDEsInAppendix}
  Y^{t,\zeta}_s=\Phi(X^{t,\zeta}_T)+\int^T_sg(r,\Theta^{t,\zeta}_r)\dif r +\int^T_sf(r,X^{t,\zeta}_r,Y^{t,\zeta}_r)\dif \eta^{t,\zeta}_r-\int^T_s\langle Z^{t,\zeta}_r,\dif B_r\rangle, \quad s\in[t,T],
\end{equation}
where $(X^{t,\zeta}_s,\eta^{t,\zeta}_s)_{s\in[t,T]}$ is the solution of RSDE \eqref{eq:RSDEsInAppendix}, the mappings $\Phi:\Omega\times\rtn^n\mapsto\rtn$ is $\calF_T$-measurable, $g:\Omega\times\tT\times\overline \calO\times\rtn\times\rtn^d$ and $f:\Omega\times\tT\times\overline \calO\times\rtn\mapsto\rtn$ are $(\calF_t)$-progressively measurable, and satisfy the following assumptions.
\begin{enumerate}
	\renewcommand{\theenumi}{(A\arabic{enumi})}
	\renewcommand{\labelenumi}{\theenumi}
	\setcounter{enumi}{4}
	\item\label{H:GFContinuousAppendix} For each $\theta\in M$, $y\mapsto g(t,\theta)$ is continuous,  $f(\omega,\cdot,\cdot,\cdot)\in C^{1,2,2}(\tT\times\rtn^n\times\rtn)$,
	\item\label{H:GFMonotonicInYLipschitzInZAppendix} There exist some constants $\lambda_1$, $\lambda_2\in\rtn$ and $K\geq 0$ such that $\pts$, $\theta$, $\theta_1$, $\theta_2\in M$,
	\begin{enumerate}
		\renewcommand{\theenumii}{(\roman{enumii})}
		\renewcommand{\labelenumii}{\theenumii}
		\item\label{A:GFMonotonicityInYAppendix} $(y_1-y_2)\big(g(t,x,y_1,z)-g(t,x,y_2,z)\big)\leq \lambda_1|y_1-y_2|^2$, 
		
		$(y_1-y_2) \big(f(t,x,y_1)-f(t,x,y_2)\big)\leq\lambda_2|y_1-y_2|^2$;
		\item\label{A:ItemGFLipschitzInXZAppendix} $|\Phi(x_1)-\Phi(x_2)|+|g(t,x_1,y,z_1)-g(t,x_2,y,z_2)|+|f(t,x_1,y)-f(t,x_2,y)|\leq K(|x_1-x_2|+|z_1-z_2|)$;
		\item\label{A:ItemGFLinearGrowthAppendix} $|\Phi(x)|+|g(t,0,y,0)|+|f(t,0,y)|\leq K(1+|x|+|y|)$.
	\end{enumerate}
\end{enumerate}
Under the previous assumptions Theorem 1.7 in \citet*{PardouxZhang1998PTRF} insures that GBSDE \eqref{eq:GBSDEsInAppendix} admits a unique solution $(Y^{t,\zeta}_s,Z^{t,\zeta}_s)_{s\in[t,T]}\in\calS^2\times\calH^2$. 

\begin{pro}\label{pro:EstimatesForSolutionsOfGBSDEsAppendix}
	Assume that \ref{H:GFContinuousAppendix} -- \ref{H:GFMonotonicInYLipschitzInZAppendix} hold. Then for each $t$, $t'\in\tT$ and $\zeta$, $\zeta'\in L^4(\Omega,\calF_{t\wedge t'},\PR;\overline \calO)$, there exists a constant $C\geq 0$ depending on $K$, $T$, $\phi$, $b$, $\sigma$, $g$ and $f$ such that $\prs$,
	\begin{equation}\label{eq:EstimateForGBSDEYZBounded}
	  \EX\bigg[\sup_{s\in[t,T]}|Y^{t,\zeta}_s|^2+\int^T_t|Z^{t,\zeta}_s|^2\dif s\bigg|\calF_t\bigg]\leq C(1+|\zeta|^2),
	\end{equation}
	\begin{equation}\label{eq:EstimateForGBSDESupHatYtLeqHatZetaHatT}
	\EX\bigg[\sup_{s\in[t\vee t',T]}|Y^{t,\zeta}_s-Y^{t',\zeta'}_s|^2+\int^T_{t\vee t'}|Z^{t,\zeta}_s-Z^{t',\zeta'}_s|^2\dif s\bigg|\calF_{t\vee t'}\bigg]\leq C(|\zeta-\zeta'|^2+|\zeta-\zeta'|+|t-t'|+|t-t'|^{1/2}),
	\end{equation}
	\begin{equation*}
	  |Y^{t,\zeta}_t-Y^{t',\zeta'}_{t'}|\leq  C(|\zeta-\zeta'|+|\zeta-\zeta'|^{1/2}+|t-t'|^{1/2}+|t-t'|^{1/4}).
	\end{equation*}
\end{pro} 

\begin{proof}
	The estimate \eqref{eq:EstimateForGBSDEYZBounded} follows from \cref{lem:EstimateForGBSDEYZDrivenByA}, \ref{A:ItemGFLinearGrowthAppendix} and \cref{pro:EstimateForRSDEHatXEta}. We focus on ourself on the second estimate. We just consider the case $t\geq t'$. To simplify presentation, set $\Psi_\cdot:=\Psi^{t,\zeta}_\cdot$, $\Psi'_\cdot:=\Psi^{t',\zeta'}_\cdot$, $\hat \Psi_\cdot:=\Psi_\cdot-\Psi'_\cdot$, where $\Psi=X$, $Y$, $Z$ and $\eta$. For two constants $\lambda$, $\mu\geq 0$ which will be chosen latter, It\^o's formula to $\me^{\lambda r+\mu\eta_r}|\hat Y_r|^2$  yields that, for each $s\in[t,T]$,
	\begin{align}
	&\me^{\lambda s+\mu\eta_s}|\hat Y_s|^2+\lambda\int^T_s\me^{\lambda r+\mu\eta_r}|\hat Y_r|^2\dif r +\mu\int^T_s\me^{\lambda r+\mu\eta_r}|\hat Y_r|^2\dif\eta_r+\int^T_s\me^{\lambda r+\mu\eta_r}|\hat Z_r|^2\dif r \nonumber\\
	&=\me^{\lambda T+\mu\eta_T} |\Phi(X_T)-\Phi(X'_T)|^2+2\int^T_s\me^{\lambda r+\mu\eta_r}\hat Y_r\big(g(r,\Theta_r)-g(r,\Theta'_r)\big)\dif r \nonumber\\
	&\quad\;+2\int^T_s\me^{\lambda r+\mu\eta_r}\hat Y_r\big(f(r,X_r,Y_r)\dif\eta_r-f(r,X'_r,Y'_r)\dif\eta'_r\big)-2\int^T_s\me^{\lambda r+\mu\eta_r}\langle\hat Z_r,\hat Y_r\dif B_r\rangle.\label{eq:ItoFormulaToHatYAppendix}
	\end{align}
	It follows from \ref{H:GFMonotonicInYLipschitzInZAppendix} and a basic inequality ($2ab\leq \alpha a^2+b^2/\alpha$ for all $\alpha>0$) that
	\begin{align*}
	2\hat Y_r\big(g(r,\Theta_r)-g(r,\Theta'_r)\big)
	\leq 2\lambda_1|\hat Y_r|^2+2K|\hat Y_r|(|\hat X_r|+|\hat Z_r|)
	\leq |\hat X_r|^2+(2\lambda_1+3K^2)|\hat Y_r|^2+\frac{1}{2}|\hat Z_r|^2,
	\end{align*}
	and, with adding and subtracting terms $f(r,X_r,Y'_r)\dif\eta_r$ and $f(r,X'_r,Y'_r)\dif\eta_r$, 
	\begin{align*}
	2\hat Y_r\big(f(r,X_r,Y_r)\dif\eta_r-f(r,X'_r,Y'_r)\big)\dif\eta'_r
	&\leq 2\lambda_2|\hat Y_r|^2\dif\eta_r+2K|\hat Y_r||\hat X_r|\dif\eta_r+2\hat Y_r f(r,X'_r,Y'_r)\dif\hat\eta_r\\
	&\leq (2\lambda_2+K^2)|\hat Y_r|^2\dif\eta_r+|\hat X_r|^2\dif\eta_r+2\hat Y_r f(r,X'_r,Y'_r)\dif\hat\eta_r.
	\end{align*}
	Plugging the previous two inequalities into \eqref{eq:ItoFormulaToHatYAppendix} and choosing $\lambda\geq 2\lambda_1+3K^2$, $\mu\geq 2\lambda_2+K^2$, we can derive that, for each $s\in[t,T]$,
	\begin{align}
	  \me^{\mu\eta_s}|\hat Y_s|^2
	  +\frac{1}{2}\int^T_s\!\me^{\mu\eta_r}|\hat Z_r|^2\dif r
	  \leq{}& \me^{\lambda T+\mu\eta_T}|\Phi(X_T)-\Phi(X'_T)|^2
	  +\!\int^T_s\!\me^{\lambda r+\mu\eta_r}|\hat X_r|^2\dif r
	  +\!\int^T_s\!\me^{\lambda r+\mu\eta_r}|\hat X_r|^2\dif\eta_r\nonumber\\
	  &+2\int^T_s\me^{\lambda r
	  +\mu\eta_r}\hat Y_rf(r,X'_r,Y'_r)\dif\hat{\eta}_r
  	-2\int^T_s\me^{\lambda r
  	+\mu\eta_r}\langle\hat Z_r,\hat Y_r\dif B_r\rangle.\label{eq:ItoFormulaToHatYAfterInnerProductLarged}
	\end{align}
	Now we estimate the last line in the previous inequality. It\^o's formula to $\me^{\lambda r+\mu\eta_r}\hat Y_rf(r,X'_r,Y'_r)\hat \eta_r$ reads that, for each $s\in[t,T]$,
	\begin{align}
	\int^T_s\me^{\lambda r+\mu\eta_r}\hat Y_rf(r,X'_r,Y'_r)\dif\hat{\eta}_r
	&=\me^{\lambda T+\mu\eta_T}\hat Y_Tf(T,X'_T,Y'_T)\hat \eta_T-\me^{\lambda s+\mu\eta_s}\hat Y_sf(s,X'_s,Y'_s)\hat \eta_s\nonumber\\
	&\hspace{-3cm}+\int^T_s\me^{\lambda r+\mu\eta_r}\hat \eta_rf^1_r\dif r+\int^T_s\me^{\lambda r+\mu\eta_r}\hat \eta_rf^2_r\dif \eta'_r+\int^T_s\me^{\lambda r+\mu\eta_r}\hat \eta_rf^3_r\dif \eta_r\nonumber\\
	&\hspace{-3cm}-\int^T_s\me^{\lambda r+\mu\eta_r}\hat Y_r\hat \eta_r\big[\langle\nabla_xf(r,X'_r,Y'_r),\sigma(r,X'_r)\dif B_r\rangle+\langle Z'_r,\partial_yf(r,X'_r,Y'_r)\dif B_r\rangle\big]\nonumber\\
	&\hspace{-3cm}-\int^T_s\me^{\lambda r+\mu\eta_r}\hat \eta_rf(r,X'_r,Y'_r)\langle \hat Z_r,\dif B_r\rangle,\label{eq:ItoFormulaToHatYFHatEta}
	\end{align}
	where we have written
	\begin{align*}
	f^1_r:={}&f(r,X'_r,Y'_r)\big(g(r,\Theta_r)-g(r,\Theta'_r)\big)
	-\langle\hat Z_r,\sigma^*(r,X'_r)\nabla_x f(r,X'_r,Y'_r)+\partial_y f(r,X'_r,Y'_r)Z'_r\rangle\\
	&-\hat Y_r\Big[\partial_r f(r,X'_r,Y'_r)+\langle\nabla_x f(r,X'_r,Y'_r),b(r,X'_r)\rangle-\partial_yf(r,X'_r,Y'_r)g(r,\Theta'_r)\\
	&\qquad\;+\frac{1}{2}Tr\{D^2_xf(r,X'_r,Y'_r)\sigma\sigma^*(r,X'_r)\}+\langle \sigma^*(r,X'_r)D^2_{xy}f(r,X'_r,Y'_r),Z'_r\rangle\\
	&\qquad\;+\frac{1}{2}\partial^2_yf(r,X'_r,Y'_r)|Z'_r|^2+\lambda f(r,X'_r,Y'_r)\Big],\\
	f^2_r:={}&\hat Y_r\big(\partial_yf(r,X'_r,Y'_r)f(r,X'_r,Y'_r)-\langle \nabla_xf(r,X'_r,Y'_r),\nabla\phi(X'_r)\rangle\big)+|f(r,X'_r,Y'_r)|^2,\\
	f^3_r:={}&(f(r,X_r,Y_r)-\mu\hat Y_r)f(r,X'_r,Y'_r).
	\end{align*}
	Since $\overline \calO$ is bounded, there exists a generic constant $C\geq 0$, which it is allowed to vary from line to line, such that $\prs$, $|\zeta|+|\zeta'|+|X_s|+|X'_s|\leq C$ for each $s\in[t,T]$. Hence, it follows from \eqref{eq:EstimateForGBSDEYZBounded} that $\prs$,
	$|Y_t|^2+|Y'_{t'}|^2\leq C$ and  $|Y^{t,\zeta}_t|\leq C(1+|\zeta|)$. Then by the uniqueness for the solutions of RSDE \eqref{eq:RSDEsInAppendix} and GBSDE \eqref{eq:GBSDEsInAppendix} we get that $\prs$, for each $s\in[t,T]$,
	\begin{equation*}
	  |Y_s|=|Y^{t,\zeta}_s|=|Y^{s,X^{t,\zeta}_s}_s|\leq C(1+|X^{t,\zeta}_s|)\leq C.
	\end{equation*} 
	Then we can deduce by \ref{A:ItemGFLipschitzInXZAppendix} and \ref{A:ItemGFLinearGrowthAppendix} that, $\!\dif\PR\!\times\!\!\dif r\text{\,--\,\,}a.e.$,
	\begin{gather*}
	  |g(r,X_r,Y_r,Z_r)|\leq K(1+|X_r|+|Y_r|+|Z_r|)\leq C(1+|Z_r|),\\
	  |f(r,X_r,Y_r)|\leq K(1+|X_r|+|Y_r|)\leq C.
	\end{gather*}
	Similarly, it holds that $\!\dif\PR\!\times\!\!\dif r\text{\,--\,\,}a.e.$, $g(r,X'_r,Y'_r,Z'_r)\leq C(1+|Z'_r|)$ and $f(r,X'_r,Y'_r)\leq C$, whence 
	$|f^1_r|\leq C(1+|Z_r|^2+|Z'_r|^2+|\hat Z_r|^2)$ and $|f^2_r|+|f^3_r|\leq C$. Thus, plugging \eqref{eq:ItoFormulaToHatYFHatEta} into \eqref{eq:ItoFormulaToHatYAfterInnerProductLarged} and taking conditional expectation with respect to $\calF_t$, then by \ref{A:ItemGFLipschitzInXZAppendix}, H\"older's inequality and \cref{pro:EstimateForRSDEHatXEta} we can deduce that,
	\begin{align*}
	  \EX\bigg[\me^{\mu\eta_s}|\hat Y_s|^2+\int^T_s\me^{\mu\eta_r}|\hat Z_r|^2\dif r\bigg|\calF_{t}\bigg]
	  \leq{}& C\bigg(\EX\bigg[\sup_{s\in[t,T]}|\hat X_s|^4\bigg|\calF_t\bigg]\bigg)^{1/2}
	  +C\bigg(\EX\bigg[\sup_{s\in[t,T]}|\hat \eta_s|^4\bigg|\calF_t\bigg]\bigg)^{1/4}\\
	  &+C\EX\bigg[\int^T_s\me^{\mu\eta_r}|\hat \eta_r|(|Z_r|^2+|Z'_r|^2)\dif r\bigg|\calF_t\bigg].
	\end{align*}
	It follows from GBSDE \eqref{eq:GBSDEsInAppendix}, BDG's inequality,  the boundedness of $X_\cdot$, $Y_\cdot$ and $f$ and \cref{pro:EstimateForRSDEHatXEta} that
	\begin{align*}
		\EX\bigg[\bigg(\int^T_s|Z_r|^2\dif r\bigg)^{2}\bigg|\calF_t\bigg]
		\leq C\EX\bigg[\sup_{r\in[s,T]}\bigg|\int^r_sZ_u\dif B_u\bigg|^4\bigg|\calF_t\bigg]
		\leq C+C(T-s)^{2}\EX\bigg[\bigg(\int^T_s|Z_r|^2\dif r\bigg)^{2}\bigg|\calF_t\bigg].
	\end{align*}
	Then for each $s\in[T-\sqrt{1/(2C)},T]$, we have that $\EX[(\int^T_s|Z_r|^2\dif r)^{2}|\calF_t]\leq C$.	Then making a partition $t=t_0<t_1<\cdots<t_N=T$ of the interval $[t,T]$ such that $t_{i+1}-t_i\leq \sqrt{1/(2C)}$, and repeating the previous arguments, we get that 
	$\EX[(\int^T_t|Z_r|^2\dif r)^{2}|\calF_t]\leq C$.	Analogous arguments yield that $\EX[(\int^T_{t}|Z'_r|^2\dif r)^2|\calF_{t}]\leq C.$ Hence, \cref{pro:EstimateForRSDEHatXEta} and H\"older's inequality indicate that 
	\begin{equation}
	  \EX\bigg[\me^{\mu\eta_s}|\hat Y_s|^2+\int^T_s\me^{\mu\eta_r}|\hat Z_r|^2\dif r\bigg|\calF_t\bigg]
	  \leq C(|\zeta-\zeta'|^2+|\zeta-\zeta'|+|t-t'|+|t-t'|^{1/2}).
	\end{equation}
	Moreover, BDG's and H\"older's inequalities yield that 
	\begin{align*}
	  \EX\bigg[\sup_{s\in[t,T]}\bigg|\int^s_t\me^{\mu\eta_r}\langle \hat Z_r,\hat Y_r\dif B_r\rangle\bigg|\bigg|\calF_t\bigg]
	  &\leq C\EX\bigg[\bigg(\int^T_t\me^{2\mu\eta_r}|\hat Y_r|^2|\hat Z_r|^2\dif r\bigg)^{1/2}\bigg|\calF_{t}\bigg]\\
	  &\leq \EX\bigg[\frac{1}{4}\sup_{s\in[t,T]}\me^{\mu\eta_s}|\hat Y_s|^2
	  +\frac{C^2}{2}\int^T_{t}\me^{\mu\eta_r}|\hat Z_r|^2\dif r\bigg|\calF_{t}\bigg],
	\end{align*}
	\begin{align*}
	  \EX\bigg[\sup_{s\in[t,T]}\bigg|\int^s_t\!\me^{\mu\eta_r}\hat \eta_rf(r,X'_r,Y'_r)\langle \hat Z_r,\!\dif B_r\rangle\bigg|\bigg|\calF_t\bigg]\!\!
	  \leq\! C\bigg(\EX\bigg[\sup_{s\in[t,T]}|\hat \eta_s|^4\bigg|\calF_t\bigg]\bigg)^{1/2}\!\!
	  +C\EX\bigg[\int^T_t\!\me^{\mu\eta_r}|\hat Z_r|^2\dif r\bigg|\calF_t\bigg],
	\end{align*}
	\begin{align*}
	  &\EX\bigg[\sup_{s\in[t,T]}\bigg|\int^s_t\me^{\mu\eta_r}\hat Y_r\hat \eta_r\big[\langle\nabla_xf(r,X'_r,Y'_r),\sigma(r,X'_r)\dif B_r\rangle
	  +\langle Z'_r,\partial_yf(r,X'_r,Y'_r)\dif B_r\rangle\big]\bigg|\bigg|\calF_t\bigg]\\
	  &\leq C\bigg(\EX\bigg[\sup_{s\in[t,T]}|\hat \eta_s|^4\bigg|\calF_t\bigg]\bigg)^{1/4}
	  +\frac{1}{4}\EX\bigg[\sup_{s\in[t,T]}\me^{\mu\eta_s}|\hat Y_s|^2\bigg|\calF_t\bigg].
	\end{align*}
	Then taking supremum with respect to $s$ in $[t,T]$ and conditional expectation with respect to $\calF_t$ on both sides of \eqref{eq:ItoFormulaToHatYAfterInnerProductLarged}, we obtain that
	\begin{align*}
	  \EX\bigg[\sup_{s\in[t,T]}\me^{\mu\eta_s}|\hat Y_s|^2\bigg|\calF_{t}\bigg]
	  &\leq
	  C\bigg(\EX\bigg[\sup_{s\in[t,T]}|\hat X_s|^4\bigg|\calF_t\bigg]\bigg)^{1/2}
	  +C\bigg(\EX\bigg[\sup_{s\in[t,T]}|\hat \eta_s|^4\bigg|\calF_t\bigg]\bigg)^{1/4}\\
	  &\qquad+C\bigg(\EX\bigg[\sup_{s\in[t,T]}|\hat \eta_s|^4\bigg|\calF_t\bigg]\bigg)^{1/2}
	  +C\EX\bigg[\int^T_t\!\me^{\mu\eta_r}|\hat Z_r|^2\dif r\bigg|\calF_t\bigg]\\
	  &\leq C(|\zeta-\zeta'|^2+|\zeta-\zeta'|+|t-t'|+|t-t'|^{1/2}).
	\end{align*}
	Therefore, we get the estimate \eqref{eq:EstimateForGBSDESupHatYtLeqHatZetaHatT}. 
	
	We proceed to estimate $|Y_t-Y'_{t'}|$. It follows from \eqref{eq:EstimateForGBSDESupHatYtLeqHatZetaHatT} that, if $t\leq t'$, 
	\[|Y_{t'}-Y'_{t'}|^2
	\leq C(|\zeta-\zeta'|^2+|\zeta-\zeta'|+|t-t'|+|t-t'|^{1/2}).\]
	For large enough $\lambda\geq0$, It\^o's formula to $\me^{\lambda r}|\overline Y_r|^2$ with $\overline Y_r:=Y_r-Y_{t'}$ in the interval $[t,t']$ yields that 
	\begin{align*}
	  &\me^{\lambda t}|\overline Y_t|^2+\lambda\int^{t'}_t\me^{\lambda r}|\overline Y_r|^2\dif r+\int^{t'}_t\me^{\lambda r}|Z_r|^2\dif r\\
	  &=2\int^{t'}_t\me^{\lambda r}\overline Y_rg(r,\Theta_r)\dif r
	  +2\int^{t'}_t\me^{\lambda r}\overline Y_rf(r,X_r,Y_r)\dif \eta_r-2\int^{t'}_t\me^{\lambda r}\langle Z_r,\overline Y_r\dif B_r\rangle.
	\end{align*}
	Then the linear growth of $g$ and the boundedness of $X_\cdot$ and $Y_\cdot$ yield that 
	\begin{equation*}
	  2\overline Y_rg(r,\Theta_r)\leq 2C|\overline Y_r|(1+|Z_r|)\leq 2C|\overline Y_r|^2+\frac{1}{2}|Z_r|^2.
	\end{equation*}
	Hence, by choosing $\lambda\geq 2C$, the boundedness of $f$ and \cref{pro:EstimateForRSDEHatXEta} we deduce that
	\begin{align*}
	  \EXlr{|\overline Y_t|^2}
	  \leq C\EX[\eta_{t'}]\leq C|t-t'|^{1/2},
	\end{align*}
	where we have used the first estimate in \cref{pro:EstimateForRSDEHatXEta}. Finally, we have that 
	\begin{align*}
	  |Y_t-Y'_{t'}|=\EXlr{|Y_t-Y'_{t'}|}
	  &\leq \EXlr{|Y_t-Y_{t'}|}+\EXlr{|Y_{t'}-Y'_{t'}|}\\
	  &\leq C(|\zeta-\zeta'|+|\zeta-\zeta'|^{1/2}+|t-t'|^{1/2}+|t-t'|^{1/4}).
	\end{align*}
	Then all the desired results are obtained. 
\end{proof}


\begin{thebibliography}{30}
	\expandafter\ifx\csname natexlab\endcsname\relax\def\natexlab#1{#1}\fi
	\expandafter\ifx\csname url\endcsname\relax
	\def\url#1{\texttt{#1}}\fi
	\expandafter\ifx\csname urlprefix\endcsname\relax\def\urlprefix{URL }\fi
	
	\bibitem[{Barles et~al.(1997)Barles, Buckdahn, and
		Pardoux}]{BarlesBuckdahnPardoux1997SSR}
	Barles, G., Buckdahn, R., Pardoux, E., 1997. Backward stochastic differential
	equations and integral-partial differential equations. Stochastics and
	Stochastics Reports 60, 57--83.
	
	\bibitem[{Bayraktar and Poor(2005)}]{BayraktarPoor2005SICON}
	Bayraktar, E., Poor, H.~V., 2005. Stochastic differential games in a
	non-{M}arkovian setting. SIAM Journal on Control and Optimization 43~(5),
	1737--1756.
	
	\bibitem[{Biswas et~al.(2017)Biswas, Ishii, Saha, and
		Wang}]{BiswasIshiiSahaWang2017SICON}
	Biswas, A., Ishii, H., Saha, S., Wang, L., 2017. On viscosity solution of {HJB}
	equations with state constraints and reflection control. SIAM Journal on
	Control and Optimization 55~(1), 365--396.
	
	\bibitem[{Briand et~al.(2000)Briand, Coquet, Hu, M\'emin, and
		Peng}]{BriandCoquetHuMeminPeng2000ECP}
	Briand, P., Coquet, F., Hu, Y., M\'emin, J., Peng, S., 2000. A converse
	comparison theorem for {BSDE}s and related properties of $g$-expectation.
	Electronic Communications in Probability 5~(13), 101--117.
	
	\bibitem[{Buckdahn et~al.(2004)Buckdahn, Cardaliaguet, and
		Rainer}]{BuckdahnCardaliaguetRainer2004SICON}
	Buckdahn, R., Cardaliaguet, P., Rainer, C., 2004. Nash equilibrium payoffs for
	nonzero-sum stochastic differential game. SIAM Journal on Control and
	Optimization 43~(2), 624--642.
	
	\bibitem[{Buckdahn and Li(2008)}]{BuckdahnLi2008SICON}
	Buckdahn, R., Li, J., 2008. Stochastic differential games and viscosity
	solutions of {H}amilton-{J}acobi-{B}ellman-{I}saacs equations. SIAM Journal
	on Control and Optimization 47~(1), 444--475.
	
	\bibitem[{Buckdahn and Nie(2016)}]{BuckdahnNie2016SICON}
	Buckdahn, R., Nie, T., 2016. Generalized {H}amilton-{J}acobi-{B}ellman
	equations with {D}irichlet boundary and stochastic exit time optimal control
	problem. SIAM Journal on Control and Optimization 54~(2), 602--631.
	
	\bibitem[{Cardaliaguet et~al.(2001)Cardaliaguet, Quincampoix, and
		Saint~Pierre}]{CardaliaguetQuincampoixSaint-Pierre2001SICON}
	Cardaliaguet, P., Quincampoix, M., Saint~Pierre, P., 2001. Pursuit differential
	games with state constraints. SIAM Journal on Control and Optimization
	39~(5), 1615--1632.
	
	\bibitem[{Crandall et~al.(1992)Crandall, Ishii, and
		Lions}]{CrandallIshiiLions1992BAMS}
	Crandall, M.~G., Ishii, H., Lions, P.-L., 1992. User's guide to viscosity
	solutions of second order partial differential equations. Bulletin of the
	American Mathematical Society 27~(1), 1--67.
	
	\bibitem[{Evans and Souganidis(1984)}]{EvansSouganidis1984IUMJ}
	Evans, L.~C., Souganidis, P.~E., 1984. Differential games and representation
	formulas for solutions of {H}amilton-{J}acobi-{I}saacs equations. Indiana
	University Mathematics Journal 33~(5), 773--797.
	
	\bibitem[{Fan et~al.(2011)Fan, Jiang, and Xu}]{FanJiangXu2011EJP}
	Fan, S., Jiang, L., Xu, Y., 2011. Representation theorem for generators of
	{BSDE}s with monotonic and polynomial-growth generators in the space of
	processes. Electronic Journal of Probability 16~(27), 830--844.
	
	\bibitem[{Fleming and Souganidis(1989)}]{FlemingSouganidis1989IUMJ}
	Fleming, W.~H., Souganidis, P.~E., 1989. On the existence of value functions of
	two-player, zero-sum stochastic differential games. Indiana University
	Mathematics Journal 38~(2), 293--314.
	
	\bibitem[{Hamad\`{e}ne and
		Lepeltier(1995{\natexlab{a}})}]{HamadeneLepeltier1995SSR}
	Hamad\`{e}ne, S., Lepeltier, J.-P., 1995{\natexlab{a}}. Backward equations,
	stochastic control and zero-sum stochastic differential games. Stochastics
	and Stochastic Reports 54~(3--4), 221--231.
	
	\bibitem[{Hamad\`{e}ne and
		Lepeltier(1995{\natexlab{b}})}]{HamadeneLepeltier1995SCL}
	Hamad\`{e}ne, S., Lepeltier, J.-P., 1995{\natexlab{b}}. Zero-sum stochastic
	differential games and backward equations. Systems and Control Letters 24,
	259--263.
	
	\bibitem[{Ikeda and Watanabe(1989)}]{IkedaWatanabe1989Book}
	Ikeda, N., Watanabe, S., 1989. Stochastic differential equations and diffusion
	processes, 2nd Edition. North-Holland Publishing Company, Tokyo.
	
	\bibitem[{Jiang(2008)}]{Jiang2008AAP}
	Jiang, L., 2008. Convexity, translation invariance and subadditivity for
	$g$-expectations and related risk measures. The Annals of Applied Probability
	18~(1), 245--258.
	
	\bibitem[{Krylov(2014)}]{Krylov2014PTRF}
	Krylov, N.~V., 2014. On the dynamic programming principle for uniformly
	nondegenerate stochastic differential games in domains and the {I}saacs
	equations. Probability Theory and Related Fields 158~(3), 751--783.
	
	\bibitem[{Leoni(2009)}]{Leoni2009Book}
	Leoni, G., 2009. A first course in {S}obolev spaces. American Mathematical
	Society, Providence, Rhode Island.
	
	\bibitem[{Li and Tang(2015)}]{LiTang2015ESAIMCOCV}
	Li, J., Tang, S., 2015. Optimal stochastic control with recursive cost
	functionals of stochastic differential systems reflected in a domain. ESAIM:
	Control, Optimisation and Calculus of Variations 21~(4), 1150--1177.
	
	\bibitem[{Li and Wei(2014)}]{LiWei2014SICON}
	Li, J., Wei, Q., 2014. Optimal control problems of fully coupled {FBSDE}s and
	viscosity solutions of {H}amilton-{J}acobi-{B}ellman equations. SIAM Journal
	on Control and Optimization 52~(3), 1622--1662.
	
	\bibitem[{Ma and Cvitani\'c(2001)}]{MaCvitanic2001JAMSA}
	Ma, J., Cvitani\'c, J., 2001. Reflected forward-backward {SDE}s and obstacle
	problems with boundary conditions. Journal of Applied Mathematics and
	Stochastic Analysis 14~(2), 113--138.
	
	\bibitem[{Mar\'in-Rubio and Real(2004)}]{Marin-RubioReal2004JTP}
	Mar\'in-Rubio, P., Real, J., 2004. Some results on stochastic differential
	equations with reflecting boundary conditions. Journal of Theoretical
	Probability 17~(3), 705--716.
	
	\bibitem[{Pardoux and Peng(1990)}]{PardouxPeng1990SCL}
	Pardoux, {\'E}., Peng, S., 1990. Adapted solution of a backward stochastic
	differential equation. Systems and Control Letters 14~(1), 55--61.
	
	\bibitem[{Pardoux and Zhang(1998)}]{PardouxZhang1998PTRF}
	Pardoux, E., Zhang, S., 1998. Generalized {BSDE}s and nonlinear {N}eumann
	boundary value problems. Probability Theory and Related Fields 110~(4),
	535--558.
	
	\bibitem[{Peng(1992)}]{Peng1992SSR}
	Peng, S., 1992. A generalized dynamic programming principle and
	{H}amilton-{J}acobi-{B}ellman equation. Stochastics and Stochastic Reports
	38~(2), 119--134.
	
	\bibitem[{Peng(1997)}]{Peng1997SuiJiFenXiXuanJiang}
	Peng, S., 1997. Backward stochastic differential equations --- stochastic
	optimization theory and viscosity solutions of hjb equations. In: Yan, J.,
	Peng, S., Fang, S., Wu, L. (Eds.), Topics on stochastic analysis (In
	Chinese). Science Press, Beijing, pp. 85--138.
	
	\bibitem[{Revuz and Yor(2005)}]{RevuzYor2005Book}
	Revuz, D., Yor, M., 2005. Continuous martingale and {B}rownian motion, 3rd
	Edition. Springer-Verlag, Berlin.
	
	\bibitem[{\'{S}wi\k{e}ch(1996)}]{Swiech1996JMAA}
	\'{S}wi\k{e}ch, A., 1996. Another approach to the existence of value functions
	of stochastic differential games. Journal of Mathematical Analysis and
	Applications 204~(3), 884--897.
	
	\bibitem[{Wu and Yu(2014)}]{WuYu2014SPA}
	Wu, Z., Yu, Z., 2014. Probabilistic interpretation for a system of quasilinear
	parabolic partial differential equation combined with algebra equations.
	Stochastic Processes and their Applications 124, 3921--3947.
	
	\bibitem[{Xiao and Fan(2017)}]{XiaoFan2017Stochastics}
	Xiao, L., Fan, S., 2017. General time interval {BSDE}s under the weak
	monotonicity condition and nonlinear decomposition for general
	$g$-supermartingales. Stochastics An International Journal of Probability and
	Stochastic Processes 89~(5), 786--816.
	
\end{thebibliography}
\end{document}